\documentclass[11pt,notitlepage,a4paper,reqno,ps]{amsart}

\usepackage{amstext}
\usepackage{hyperref}

\usepackage{paralist}

\usepackage{stmaryrd,mathrsfs,bm,amsthm,mathtools,yfonts,amssymb,esint}

\newcommand{\N}{\mathbb{N}}
\newcommand{\medint}{-\kern  -,375cm\int}

\theoremstyle{plain}
\newtheorem{theorem}{Theorem}[section]

\newtheorem{lemma}[theorem]{Lemma}
\newtheorem{proposition}[theorem]{Proposition}
\theoremstyle{definition}
\newtheorem{definition}[theorem]{Definition}
\theoremstyle{remark}
\newtheorem{remark}[theorem]{Remark}
\theoremstyle{plain}

\def\eps{\varepsilon}

\def\div{{\rm div}}

\newcommand{\R}{\mathbb R}
\newcommand{\Rn}{\mathbb{R}^n}

\newcommand{\cL}{\mathcal{L}}
\newcommand{\Ln}{\mathcal{L}^n}

\newcommand{\Om}{\Omega}

\newcommand{\zz}{\underline{0}}
\newcommand{\A}{\mathbb {A}}
\newcommand{\LL}{\mathbb {L}}
\newcommand{\HH}{{\mathcal H}^{n-1}}
\newcommand{\EEE}{\mathscr{E}}
\newcommand{\HHH}{\mathscr{H}}
\newcommand{\GGG}{\mathscr{G}}
\newcommand{\KK}{\mathcal K}
\newcommand{\Id}{\mathrm{I}_{n}}

\newcommand{\de}{\partial}

\renewcommand{\div}{\textup{div}}

\newcommand\Sing{\textup{Sing}}
\newcommand\Reg{\textup{Reg}}

\numberwithin{equation}{section} \makeatletter

\renewcommand{\p@enumi}{\thesection.}
\makeatother  \allowdisplaybreaks
\begin{document}
\title[Monotonicity formulas for Lipschitz obstacle problems]{Monotonicity formulas for obstacle problems with 
Lipschitz coefficients}

\author[M. Focardi, M.S. Gelli, E. Spadaro]{M. Focardi, M. S. Gelli and E. Spadaro}

\address{DiMaI
``U. Dini'', Universit\`a di Firenze,
V.le Morgagni 67/A, I-50134 Firenze}
\email{focardi@math.unifi.it}

\address{Dipartimento di Matematica
``L. Tonelli'', Universit\`a di Pisa,
L.go Pontecorvo 5, I-56127 Pisa}
\email{gelli@dm.unipi.it}

\address{Max-Planck-Institut f\"ur Mathematik in den Naturwissenschaften, 
Inselstrasse 22, D-04103 Leipzig} 
\email{spadaro@mis.mpg.de}

\begin{abstract}
We prove quasi-monotonicity formulas for classical obstacle-type problems 
with energies being the sum of a quadratic form with
Lipschitz coefficients, and a H\"older continuous linear term.
With the help of those formulas we are able to carry out the full analysis of
the regularity of free-boundary points following the approaches in
\cite{Ca98,Monneau,Weiss}.

\medskip

\end{abstract}

\maketitle

\section{Introduction}
In this note we extend the regularity theory for the obstacle problem to the case of
quadratic energies with Lipschitz coefficients.
The obstacle problem is a well-known topic in partial
differential equations and, in its classical formulation,
consists in finding the equilibrium solution for a scalar order parameter
$u$ constrained to lay above a given obstacle, $u\geq \psi$ -- see, e.g.,
\cite{Frie,Ro} for several applications in physics.
The analytical interests in this kind of problems are mostly related to the
study of the properties of the \textit{free boundary}, the boundary of the set
where the equilibrium configuration touches the obstacle.
This subject has been developed over the last 40 years by the
works of many authors; it is not realistic to give here a complete account:
we rather refer to the textbooks \cite{CS, Frie, KS, PSU, Ro}
for a fairly vast bibliography and its historical developments.

Very recently many authors have drawn the attention on the issue of
weakening the hypotheses on the operators governing the obstacle-type problems,
in order to enlarge the applicability of the results and deepen the analytical
techniques introduced in the study of such problems (cp.~\cite{CFS, FeSa07, FeSa10, MaPe, Wang00, Wang02}).
The prototype result in obstacle-type problems is a stratification
of the free boundary $\de \{u = \psi\}$ in terms of the properties of
the blowup limits.
%  i.e.~$\de \{u = \psi\} = \cup_{k=0}^{n-1} M_k$, where
% each $M_k$ is contained in the union of countably many $k$-dimensional
% submanifolds of class $C^1$ -- we refer to the statements below for
% a more precise formulation.

In this note we complete this program for the case of an obstacle problem
with a quadratic energy having Lipschitz coefficients and 
suitable obstacle functions $\psi$
(e.g., such that $\div (\A \nabla \psi) \in C^{0,\alpha}$ in the distributional
sense), which can be reduced to the $0$ obstacle case.
We collect in the 
statement below the main results of our analysis, in particular the contents 
of Theorems~\ref{t:regular} and \ref{t:singular}.

\begin{theorem}\label{t:main}
Let $\Om\subset\Rn$ be smooth, bounded and open,
$\A\in \mathrm{Lip}(\Om,\R^{n\times n})$ be symmetric and uniformly elliptic,
i.e.~$\lambda^{-1}|x|^2\leq\langle \A(x) x,x\rangle\leq \lambda|x|^2$
for all $x \in \R^n$, and $f\in C^{0,\alpha}(\Om)$
for some $\alpha\in(0,1]$ and $f\geq c_0 >0$.
Let $u$ be the solution of the obstacle problem
\[
\min \EEE[v]:=\int_\Omega\big(\langle\A(x)\nabla v(x),\nabla v(x)\rangle
+2 f(x)\,v(x)\big)\,dx,
\]
where the minimum is taken in
\[
\KK := \big\{v\in H^1(\Om):\,v\geq 0\;\;\Ln\,\text{-a.e. on } \Om,\,
\textup{Tr}(v)=g \,\text{ on } \partial\Om\big\},
\]
for $g\in H^{1/2}(\partial\Om)$ a nonnegative function.
Then, $u$ is $C^{1,\gamma}_{loc}$ regular in $\Om$
for every $\gamma \in (0,1)$, and the free boundary decomposes as
$\de \{u = 0\} \cap \Omega = \Reg(u) \cup \Sing(u)$, where
$\Reg(u) \cap \Sing(u) = \emptyset$ and
\begin{itemize}
\item[(i)] $\Reg(u)$ is relatively open in $\de \{u = 0\}$
and, for every point $x_0 \in \Reg(u)$, there exist $r=r(x_0)>0$
and $\beta= \beta(x_0) \in (0,1)$ such that
$\Reg(u) \cap B_r(x_0)$ is a $C^{1,\beta}$ submanifold of dimension $n-1$;
\item[(ii)] $\Sing(u) = \cup_{k=0}^{n-1} S_k$, with $S_k$ contained in the union of
at most countably many submanifold of dimension $k$ and class $C^1$.
\end{itemize}
\end{theorem}

The theorem above for the Dirichlet energy is the outcome of a long term program and 
of the efforts of many authors.
It has been proved first by Caffarelli \cite{Ca98} under more restrictive hypothesis 
on $f$, namely $f \in C^{1,\alpha}$.
The proof in \cite{Ca98} is based on a monotonicity formula introduced by
Alt, Caffarelli and Friedman \cite{ACF} and on the regularity of harmonic
functions in Lipschitz domains \cite{AtCa, CFMS, JK}.
Since then, different approaches have been introduced, most remarkably the variational
one by Weiss \cite{Weiss} and Monneau \cite{Monneau}, who extended the techniques to 
deal with H\"older continuous linear terms $f$ and simplified the arguments
for the analysis of the free boundary.
These improvements allowed to extend the results by Caffarelli to some
other obstacle-type settings, such as the no-sign obstacle problem \cite{CKS}
and the two-phases membrane problem \cite{Weiss01} --
see \cite{PSU} for more detailed comments, and \cite{Lin} for a revisitation
of such arguments in a geometric measure theory flavour.

The lack of regularity and homogeneity of the coefficients in our framework
does not allow us to exploit any simple freezing argument in a way to
reduce the problem to the ones above for regular operators.
Indeed, in the proof of Theorem~\ref{t:main} we take advantages of the full
strength of those contributions, including the remarkable
epiperimetric inequality established by Weiss \cite{Weiss}.
We prove quasi-monotonicity formulas analogous to those
introduced by Weiss and Monneau for the Laplace equation.
To this aim, we exploit some intrinsic computations based on a generalization of 
Rellich and Ne\v{c}as' identity due to Payne and Weinberger
(which we first learned by Kukavica \cite{Kukavica}).

Our results leads to the stratification of the free boundary for more general
obstacle problems with quasi-linear operator with $C^{1,1}$ regular solutions:
\[
\min \int_\Omega \left( F(|\nabla u|^2) + G(x,u) \right) dx,
\]
with $F, G$ satisfying suitable assumptions,
as, e.g., the ones considered in \cite{Monneau}, which covers the
case of the area functional.
In particular, we point out the recent contribution by Matevosyan and Petrosyan
\cite{MaPe}, where they perform the analogous improvement of the ACF monotonicity
formula for more general operators.
As a byproduct of their analysis, $C^{1,1}$ regularity of
solutions of a broad class of obstacle problems follows and,
combining these results with our analysis, the complete stratification
of the free boundary may be inferred for classical obstacle problems
corresponding to a subclass of the quasi-linear operators considered by these authors,
with applications to certain mean-field models for type II superconductors (cp., e.g., \cite{FeSa10, MaPe}).

To conclude this introduction we describe briefly the contents of the paper: Section~\ref{s:prel} 
is devoted to settle the notations, fix the main assumptions and derive the first basic results 
on the problem. Weiss' and Monneau's quasi-monotonicity formulas are then established 
in Section~\ref{s:wm} (cp.~with Theorems~\ref{t:Weiss} and \ref{t:Monneau}, respectively). 
The latter are instrumental tools to study in Section~\ref{s:freeboundary} the blow-up limits in 
free boundary points (cp.~with Propositions~\ref{p:blowup}, \ref{p:classification}, \ref{p:uniqueness blowup reg}
and \ref{p:uniform continuity}). 
In turn, such an analysis leads to the regularity results stated in Theorem~\ref{t:main}
(cp. with~Theorems\ref{t:regular} and \ref{t:singular}).

\section{Preliminaries}\label{s:prel}
Let $\Om\subset\Rn$ be a smooth, bounded and open set. 
Let $\A:\Om\to\R^{n\times n}$ be a matrix-valued field and 
$f:\Om\to\R$ be a function satisfying 
\begin{itemize}
\item[(H1)] $\A\in \mathrm{Lip}(\Om,\R^{n\times n})$; 
\item[(H2)] $\A(x)=(a_{ij}(x))_{i,j=1}^n$ is symmetric and coercive, 
i.e.~$a_{ij}=a_{ji}$ and, for some $\lambda\geq 1$,
\[
\lambda^{-1}|x|^2\leq\langle \A(x) x,x\rangle\leq \lambda|x|^2
\quad\mbox{for all $x\in\R^n$};
\]
\item[(H3)] $f\in C^{0,\alpha}(\Om)$ for some $\alpha\in(0,1]$ and $f\geq c_0 >0$.
\end{itemize}
\begin{remark}
For some of the results of the paper, a weaker condition on $f$ would suffice 
(e.g., a continuous function with a modulus of continuity satisfying a certain 
Dini-type integrability condition -- cp. \cite{PeSh}). 
However, we do not pursue this issue here.
\end{remark}
For all open subsets $A$ of $\Om$ and functions $v\in H^1(\Om)$, we consider 
the energy
\begin{equation}\label{e:enrg}
\EEE[v,A]:=\int_A\big(\langle\A(x)\nabla v(x),\nabla v(x)\rangle
+2 f(x)\,v(x)\big)\,dx,
\end{equation}                      
and the related minimum problem $\inf_\KK\EEE[\cdot,\Om]$, where 
$\KK$ is the weakly closed convex subset of $H^1(\Omega)$ given by
\[
\KK := \big\{v\in H^1(\Om):\,v\geq 0\;\;\Ln\,\text{-a.e. on } \Om,\,
\textrm{Tr}\,v=g \,\text{ on } \partial\Om\big\},
\]
with $g\in H^{1/2}(\partial\Om)$ a nonnegative function.

Existence and uniqueness for the above minimum problem follow straightforwardly 
from (H1)-(H3).
In fact, the energy $\EEE$ is coercive and strictly convex in $\KK$, which implies
the lower semicontinuity for the weak topology in $H^1(\Omega)$ and 
the uniqueness of the minimizer, denoted in the sequel by $u$. 
Moreover, letting for any $v\in H^1(\Omega)$, 
\begin{equation}\label{e:enrgG}
\GGG[v,\Om]:=\int_\Omega\big(\langle\A(x)\nabla v(x),\nabla v(x)\rangle
+2 f(x)\,v^+(x)\big)\,dx,
\end{equation}
we easily infer the existence of a unique minimizer for $\GGG$ on $H^1(\Om)$
with boundary trace equal to $g\in H^{1/2}(\partial\Om)$ and satisfying
\[
\min_{\KK}\EEE[\cdot,\Om]=\min_{g+H^1_0(\Om)}\GGG[\cdot,\Om].
\] 
As in the classical case, the minimizer $u$ satisfies a PDE both 
in a distributional sense and almost everywhere in $\Omega$, as pointed out in the 
next proposition. 

\begin{proposition}\label{p:pde}
Let $u$ be the minimizer of $\EEE$ in $\KK$. Then,
\begin{equation}\label{e:pdeu}
\div\left(\A(x)\nabla u(x)\right)=f(x)\,\chi_{\{u>0\}}(x)
\quad\text{a.e. in $\Omega$ and in }\mathcal{D}^\prime(\Om).
\end{equation}
\end{proposition}

\begin{proof}
Let $\varphi\in H^1_0\cap C^0 (\Omega)$ 
and $\eps >0$, and consider $u+\eps\varphi$ as a competitor for $\GGG$. Then,
\begin{align}\label{e:minim}
0 \le & \eps^{-1} \Big(\GGG[u +\eps\varphi,\Om]-\GGG[u,\Om]\Big) \notag\\
= & \int_\Omega\big(\eps\langle\A\nabla \varphi,\nabla \varphi\rangle
+2\langle\A\nabla u,\nabla \varphi\rangle\big)\,dx 
+2\,\eps^{-1}\int_\Omega f\,\big((u+\eps\varphi)^+-u\big)\,dx,
\end{align}   
where in the last identity we have used the positivity of $u$. 
Expanding the computation, we get
\begin{equation}\label{e:stima}
\int_\Omega f\,\big((u+\eps\varphi)^+-u\big)\,dx= 
\eps\int_{\{u+\eps\varphi\ge 0\}} f\,\varphi\,dx 
-\int_{\{u+\eps\varphi< 0\}} f\,u\,dx. 
\end{equation} 
We note that 
\[
0\leq\int_{\{u+\eps\varphi< 0\}} f\,u\,dx\leq 
-\eps\int_{\{u+\eps\varphi< 0\}} f\,\varphi\,dx=o(\eps).
\]
Moreover, setting $A_\varphi := \{u=0\}\cap \{\varphi\geq0\}$, it is easy to
show that $\chi_{\{u+\eps\varphi\ge 0\}} \to \chi_{A_\varphi \cup \{u>0\}}$
in $L^1$. Then, passing into the limit in \eqref{e:minim}, by \eqref{e:stima} and 
the dominated convergence theorem, we deduce that
\begin{equation}\label{e:distribuzione}
\int_\Omega \langle\A\nabla u,\nabla \varphi\rangle\,dx 
+\int_\Omega \varphi\,f\,\chi_{\{u>0\}\cup A_{\varphi}}\,dx\geq 0.
\end{equation}
Set now
\[
T(\varphi) := \int_\Omega \langle\A\nabla u,\nabla \varphi\rangle\,dx 
+\int_\Omega \varphi\,f\,\chi_{\{u>0\}}\,dx.
\]
By applying \eqref{e:distribuzione} with $\pm \varphi$, we deduce
\begin{equation}\label{e:misura}
- \int_{A_\varphi} \varphi\,f\,dx \leq T(\varphi) \leq 
- \int_{\{u = 0\}\cap \{\varphi \leq 0\}} \varphi\,f\,dx.
\end{equation}
In particular, by a density argument, we deduce that 
\[
|T(\varphi)| \leq C\, \|\varphi\|_{L^\infty(\Omega)}\quad 
\textrm{for every $\varphi \in C^0_0(\Omega)$}.
\]
This, in turn, implies that the distribution $T$ is a (nonpositive) Borel measure
which, in view of \eqref{e:misura}, is dominated by an absolutely continuous measure
with respect to the Lebesgue measure, so that $T = \zeta \,dx$ for some density 
$\zeta \in L_{\textup{loc}}^1(\Omega)$.
Moreover, again by \eqref{e:misura}, we deduce that $\zeta = 0$ $\cL^n$ a.e. on
$\{u>0\}$; and, since $\nabla u = 0$ $\cL^n$ a.e. on the set $\{u=0\}$,
by the very definition of $T$ we also get $\zeta=0$ $\cL^n$ a.e. in $\Omega$.
Clearly, this shows \eqref{e:pdeu}.
% that
% \begin{equation*}
% \div (\A\,\nabla u ) = f\, \chi_{\{u>0\}} \quad \text{a.e. in } \; \Omega,\,\,
% \text{and in }\mathcal{D}^\prime(\Om).\qedhere
% \end{equation*}
\end{proof}

The regularity theory for uniformly elliptic equations with Lipschitz coefficients 
(cp.~\cite[Chapter~III, Theorem~3.5]{Gi}) and Sobolev embeddings yield that 
$u \in W^{2,p}_{\textup{loc}}(\Omega)$ for every $p\in[1,\infty)$, and hence
$u \in C_{\textup{loc}}^{1,\gamma}(\Omega)$ for every $\gamma \in (0,1)$.
Note that, contrary to the usual obstacle-type problems, in general $u$ fails to be
$C_{\textup{loc}}^{1,1}$, because of the lack of regularity of the coefficients $\A$
(see \cite[Exercise~4.9]{GT} for a related counterexample). Despite this, the sign
condition on $u$ guarantees $C^{1,1}$ regularity on the set $\{u=0\}$ (cp. with
Proposition~\ref{p:C11} below).

Finally, we fix the notation for the \emph{coincidence set}, the 
\emph{non-coincidence set} and the \emph{free boundary}:
\[
\Lambda_u := \{u =0\}, \quad
N_u :=\{u>0\} \quad \text{and}\quad
\Gamma_u:= \de \Lambda_u \cap \Omega.
\]

\section{Weiss' and Monneau's quasi-monotonicity formulas}\label{s:wm}
In this section we show that the monotonicity formulas 
established by Weiss \cite{Weiss} and Monneau \cite{Monneau} in the standard 
case of the Laplace operator, i.e. $\A\equiv\Id$, hold in an approximate way
in the present setting.

\subsection{Notation and preliminary results}
The first step towards the monotonicity formulas is to fix appropriate systems
of coordinates with respect to which the formulas will be written.
Let $x_0 \in \Gamma_u$ be any point of the free boundary, then the affine change
of variables
\[
x\mapsto x_0+f^{-1/2}(x_0)\A^{1/2}(x_0)x =: x_0 + \LL(x_0)\, x
\]
leads to
\begin{equation}\label{e:cambio di coordinate3}
\EEE[u,\Om]=f^{1-\frac{n}{2}}(x_0)\det(\A^{1/2}(x_0))\,\EEE_{\LL(x_0)}
[u_{\LL(x_0)},\Om_{\LL(x_0)}],
\end{equation}
where, for all open subset $A$ of $\Om_{\LL(x_0)}:=\LL(x_0)\,(\Om-x_0)$, we set
\begin{equation}\label{e:enrgA}
\EEE_{\LL(x_0)}[v,A]:=
\int_{A}
\left(\langle \mathbb{C}_{x_0}\nabla v,\nabla v\rangle 
+ 2\frac{f_{\LL(x_0)}}{f(x_0)}\,v\right)dx,
\end{equation}
with
\begin{gather}
u_{{\LL(x_0)}}(x)  :=u\big(x_0+\LL(x_0)x\big), 
\label{e:cambio di coordinate1}\\
f_{{\LL(x_0)}}(x)  :=f\big(x_0+\LL(x_0)x\big), 
\label{e:cambio di coordinate2}\\
{\mathbb C}_{x_0}(x)  :=
\A^{-1/2}(x_0)\A(x_0+\LL(x_0)x)\A^{-1/2}(x_0).
\end{gather}
% and the matrix field 
% ${\mathbb C}_{x_0}\in\mathrm{Lip}(\Om_{\A(x_0)},\R^{n\times n})$ 
% is given by
% \[
% {\mathbb C}_{x_0}(x):=
% \A^{-1/2}(x_0)\A(x_0+f^{-1/2}(x_0)\A^{1/2}(x_0)x)\A^{-1/2}(x_0).
% \]
Note that $f_{\LL(x_0)}(\zz)=f(x_0)$ and ${\mathbb C}_{x_0}(\zz)=\Id$.
Moreover, the free boundary is transformed under this map into
\[
\Gamma_{u_{\LL(x_0)}}=\LL(x_0)(\Gamma_u-x_0),
\] 
and the energy $\EEE$ in \eqref{e:enrg} is minimized by $u$ if and only if 
$\EEE_{\LL(x_0)}$ in \eqref{e:enrgA} is minimized by the function $u_{\LL(x_0)}$
in \eqref{e:cambio di coordinate1}.

Therefore, for a fixed base point $x_0\in\Gamma_u$, we change the coordinates 
system in such a way that (with a slight abuse of notation we do not rename the 
various quantities) we reduce to
\begin{equation}\label{e:rinormalizzazione}
x_0 = \zz \in \Gamma_u, \quad \A(\zz) = \Id \quad\text{and}\quad f(\zz) = 1.
\end{equation}
This convention shall be adopted throughout this section to simplify
the ensuing calculations. Note that with this convention at hand $\zz\in\Om$.
In this new system of coordinates we define
\[
\nu(x):=\frac x{|x|}\quad\text{for $x\neq\zz$}\,,
\]
and
\begin{equation}\label{e:mu}
\mu(x):=\langle \A(x)\,\nu(x),\nu(x)\rangle
\quad\text{for $x\neq\zz$},\quad\mu(\zz):=1.
\end{equation}
Note that $\mu\in C^0(\Om)$ thanks to (H1) and \eqref{e:rinormalizzazione}.
Actually, $\mu$ is Lipschitz continuous.

\begin{lemma}\label{l:mulip}
If $\A$ satisfies (H1)-(H2) and \eqref{e:rinormalizzazione}, then
$\mu\in C^{0,1}(\Om)$, and
\begin{equation}\label{e:mulip}
|\mu(x)-\mu(y)|\leq C\,\|\A\|_{W^{1,\infty}}|x-y|\quad\text{for all }x,y\in\Om,
\end{equation}
where $C>0$ is a dimensional constant $C>0$, and
\begin{equation}\label{e:mubdd}
\lambda^{-1}\leq\mu(x)\leq\lambda\quad\text{for all }x\in\Rn.
\end{equation}
\end{lemma}

\begin{proof}
Note that in case $y=\zz$ we have
\[
\mu(x)-\mu(\zz)=\langle(\A(x)-\Id)\frac x{|x|},\frac x{|x|}\rangle,
\]
so that estimate \eqref{e:mulip} follows directly from (H1).

Then let $x, y\neq\zz$ and
%We start off noting that we may assume $|x|=|y|$, otherwise we
set $z=|y|\frac x{|x|}$. Then, $|z|=|y|$ and by triangle inequality
\begin{align*}
|\mu(x)-\mu(y)|&\leq|\mu(x)-\mu(z)|+|\mu(z)-\mu(y)|
\leq |\A(x) - \A(z)| + |\mu(z)-\mu(y)|\\
&\leq
\|\A\|_{W^{1,\infty}}||x|-|y|| + |\mu(z)-\mu(y)|.
\end{align*}
We need only to estimate the last term.
Set for simplicity $|z|=|y|=r$, and use again the triangle inequality
\[%\begin{equation}\label{e:mu1}
|\mu(z)-\mu(y)|\leq
|\langle(\A(z)-\A(y))\frac z{r},\frac z{r}\rangle|
+|\langle\A(y)\frac z{r},\frac z{r}\rangle
-\langle\A(y)\frac y{r},\frac y{r}\rangle|.
\]%\end{equation}
The first term is easily estimated thanks to (H1), 
\begin{equation}\label{e:mu2}
|\langle(\A(z)-\A(y))\frac z{r},\frac z{r}\rangle|
\leq\|\A\|_{W^{1,\infty}}|z-y|.
\end{equation}
For the second term  
%to conclude we need only to estimate the last term. To this aim
we use equality $\A(\zz)=\Id$ (see \eqref{e:rinormalizzazione}) and $|z|=|y|=r$ 
to rewrite it as follows:
\[
\langle\A(y)\frac z{r},\frac z{r}\rangle
-\langle\A(y)\frac y{r},\frac y{r}\rangle=
\langle \A(y)\frac {z+y}{r},\frac {z-y}{r}\rangle=
\langle(\A(y)-\A(\zz))\frac {z+y}{r},\frac {z-y}{r}\rangle,
\]
which in turn implies
\begin{equation}\label{e:mu3}
|\langle\A(y)\frac z{r},\frac z{r}\rangle
-\langle\A(y)\frac y{r},\frac y{r}\rangle|
\leq 2\|\A\|_{W^{1,\infty}}|z-y|.
\end{equation}
Since $|z-y|\leq |z-x|+|x-y|\leq 2|x-y|$,
inequalities \eqref{e:mu2} and \eqref{e:mu3} yield \eqref{e:mulip}.
Estimate \eqref{e:mubdd} follows easily from (H2).
\end{proof}

Next we introduce the following notation for the rescaled functions and the 
rescaled energies:
\begin{equation}\label{e:ur}
u_{r}(x):=\frac{u(rx)}{r^2},
\end{equation}
\begin{align}\label{e:E}
\EEE(r) :={}& \EEE[u,B_r]=\int_{B_r}\left(\langle\A(x)\nabla u(x),\nabla u(x)\rangle
+2f\,u\right)\,dx\notag\\
={}& r^{n+2}\int_{B_1}\left(\langle\A(rx)\nabla u_r(x),\nabla u_r(x)\rangle
+2f(rx)\,u_r(x)\right)\,dx,
\end{align}
\begin{equation}%\begin{align}
\label{e:H}
\HHH(r)  := %{}&
\int_{\partial B_r}\mu\,u^2d\HH%\notag\\={}& 
=r^{n+3} \int_{\partial B_1}\mu(rx)\,u_r^2(x)d\HH,
\end{equation}%\end{align}
and
\begin{equation}\label{e:phi}
\Phi(r):=r^{-n-2}\EEE(r)-2\,r^{-n-3}\HHH(r).
\end{equation}

%\medskip
Although the minimizer $u$ is not in general globally $C^{1,1}_{\textup{loc}}$,
the rescaled functions $u_r$ satisfy uniform $W^{2,p}_{\textup{loc}}$ 
estimates thanks to Harnack inequality.
\begin{proposition}\label{p:C11}
Let $u$ be the solution to the obstacle problem \eqref{e:enrg},
and assume \eqref{e:rinormalizzazione} holds.
Then, for every $p\in [1,\infty)$ and $R>0$, there exists a constant 
$C= C(p,R)>0$ such that, for every $r\in(0,\frac1{2R}\textup{dist}(\zz,\de\Omega))$,
\begin{equation}\label{e:C11}
\|u_r\|_{W^{2,p}(B_R)} \leq C.
\end{equation}
In particular, the functions $u_r$ are equibounded in
$C^{1,\gamma}_{\textup{loc}}(\R^n)$ for every $\gamma\in (0,1)$.
\end{proposition}

\begin{proof}
By Proposition~\ref{p:pde} and \eqref{e:ur}, we have that
$\div (\A(rx) \nabla u_r(x)) = f(rx) \chi_{\{u_r>0\}}(x)$ in the weak sense.
Since $u$ is non-negative, we can apply the Harnack inequality
(cp.~\cite[Theorems~8.17 and 8.18]{GT}) to infer that, for a positive
constant $C=C(n,\lambda)$, 
\[
\|u_r\|_{L^\infty(B_{2R})} \leq C\,\|f\|_{L^\infty(B_{2R})}.
\]
Let now $w$ be the harmonic function with 
$w\vert_{\de B_{2R}} = u_r\vert_{\de B_{2R}}$, and
\[
g_r(x):= f(rx) \chi_{\{u_r>0\}}(x) - r \nabla\A(rx) \nabla w(x) - \A(rx):\nabla^2w(x),
\]
where $:$ stands for the scalar product between $n \times n$ matrices.
As $\|g_r\|_{L^\infty(B_{R})} \leq C$ uniformly in $r$, and
\[
\div (\A(rx)\nabla (u_r-w)(x))=g_r(x)
\]
by elliptic regularity theory (cp.~\cite[Chapter~III Theorem~3.5]{Gi},
\cite[Chapter~IV Theorem~A.1]{KS}),
we deduce that
\[
\|u_r\|_{W^{2,p}(B_R)} \leq \|u_r-w\|_{W^{2,p}(B_R)} +\|w\|_{W^{2,p}(B_R)} 
\leq C\,\|g_r\|_{L^\infty(B_{2R})}+C\leq C.\qedhere
\]
\end{proof}
\begin{remark}\label{r:Ograndi}
We recall for later reference the following indentities inferred from 
the definitions in \eqref{e:E}, \eqref{e:H} and Propositon \ref{p:C11}:
\begin{multline}\label{e:freezing1}
\EEE(r) = \int_{B_r} \big( |\nabla u|^2 + 2\, u \big) + O(r^{n+2 + \alpha}),
\quad
\HHH(r) = \int_{\de B_r} u^2 d\HH + O(r^{n+4}),\\
\int_{\de B_r} \big(\langle \A\,\nabla u, \nabla u\rangle + 2\,f\,u \big)d\HH
= \int_{\de B_r} \big(|\nabla u|^2 + 2\,u \big)d\HH + O(r^{n+1+\alpha}).
\end{multline}
Moreover, we have from \eqref{e:C11} that
\begin{equation}\label{e:scaling}
\EEE(r) = O(r^{n+2}) \quad\text{and}\quad \HHH(r) = O(r^{n+3}),
\end{equation}
and, since $u(\zz)=0$ and $\nabla u(\zz)=\zz$,
\begin{equation}\label{e:c11fb}
\|u\|_{L^\infty(B_r)}\leq C\,r^2\quad\text{and}\quad
\|\nabla u\|_{L^\infty(B_r,\Rn)}\leq C\,r.
\end{equation}
Note that the constant $C$ in \eqref{e:c11fb} depends only on the constant
in \eqref{e:C11} and, therefore, is uniformly bounded for points
$x_0\in \Gamma_u \cap K$, for any compact $K \subset \Omega$.
\end{remark}

\subsection{Derivatives of $\mathscr{E}$ and $\mathscr{H}$}
We provide next some estimates for the derivatives of $\EEE$ and $\HHH$.
To this aim, we have benefited of some insights developed in \cite{Kukavica},
concerning Payne-Weinberger's generalization of Rellich-Ne\v{c}as' 
identity. The symbol $:$ below denotes the scalar product in the space of 
third order tensors.

\begin{lemma}\label{l:pw}
Let ${\bf F}\in W^{1,\infty}(B_r,\Rn)$. Then, for every $w\in W^{2,p}(\Om)$, 
$p\in[\frac{2n}{n+1},\infty)$, it holds 
\begin{align}\label{e:pw}
\int_{\partial B_r}&\big(\langle\A\nabla w,\nabla w\rangle
\langle {\bf F},\nu\rangle
-2\langle\A\,\nu,\nabla w\rangle\langle {\bf F},\nabla w\rangle\big)d\HH \notag
\\
&=\int_{B_r}\big(\langle\A\nabla w,\nabla w\rangle\div {\bf F}
-2\langle {\bf F},\nabla w\rangle\div(\A\nabla w) \big)dx\notag\\
& \quad +\int_{B_r}\left( \nabla\A:{\bf F}\otimes\nabla w\otimes\nabla w
-2\langle\A\nabla w,\nabla^T{\bf F}\nabla w\rangle\right)dx.
\end{align}
\end{lemma}

\begin{proof}
The proof is a direct application of the Divergence theorem and
the expansion of
\[
\div\Big(\langle\A\nabla w,\nabla w\rangle\, {\bf F}
-2\,\langle {\bf F},\nabla w\rangle\A\nabla w\Big).\qedhere
\] 
\end{proof}

In particular, we can compute the derivative of the energy $\EEE$ on balls as follows. 

\begin{proposition}\label{p:eprime}
There exists a non negative constant $C_1$ depending on $\lambda$, 
and on the Lipschitz constant of $\A$, such that, for $\cL^1$ 
a.e. $r\in(0,\mathrm{dist}(\zz,\partial\Om)))$,
\begin{align}\label{e:eprime}
\EEE^\prime(r) = {} &
2\int_{\partial B_r} \mu^{-1}\langle\A\,\nu,\nabla u\rangle^2d\HH
+\frac 1r\int_{B_r}\langle\A\nabla u,\nabla u\rangle\,
\div\left(\mu^{-1}\A x\right)\,dx
-\frac 2r\int_{B_r}f\,\langle \mu^{-1}\A x,\nabla u\rangle\,dx\notag\\
&-\frac 2r\int_{B_r}
\langle\A\nabla u,\nabla^T\left(\mu^{-1}\A x\right)\nabla u\rangle\,dx
+2\int_{\partial B_r}f\,u\,d\HH + \eps(r),
\end{align} 
with $|\eps(r)| \leq C_1\,\EEE(r)$.
\end{proposition}
\begin{proof}
Consider the vector field
\[
{\bf F}(x):=\frac{\A(x)x}{r\mu(x)}.
\]
{\bf F} is admissible for Lemma~\ref{l:pw} because of (H1) and 
Lemma~\ref{l:mulip}. Simple computations shows that 
\[
\langle {\bf F},\nu\rangle=1 \quad\text{on }\partial B_r
\quad \text{and}\quad
\langle {\bf F},\nabla u\rangle= \mu^{-1} \langle\A\,\nu,\nabla u\rangle
\quad\text{on }\partial B_r.
\]
By the coarea formula, for $\cL^1$ a.e.~$r\in(0,\mathrm{dist}(\zz,\partial\Om))$,
it holds
\[
\EEE^\prime(r)=\int_{\partial B_r}
(\langle\A\nabla u(x),\nabla u(x)\rangle+2f\,u)\,d\HH.
\]
Lemma~\ref{l:pw}, with the above choice of ${\bf F}$ and \eqref{e:pdeu}, yields
\begin{multline*}
\EEE^\prime(r)=
2\int_{\partial B_r}\mu^{-1}\langle\A\,\nu,\nabla u\rangle^2d\HH
+\frac 1r\int_{B_r}\langle\A\nabla u,\nabla u\rangle\,
\div\left(\mu^{-1}\A x\right)\,dx
-\frac 2r\int_{B_r}f\,\langle \mu^{-1}\A x,\nabla u\rangle\,dx\\
+\frac 1r\int_{B_r}\mu^{-1}(\nabla\A:\A x\otimes\nabla u\otimes\nabla u)\,dx
-\frac 2r\int_{B_r}
\langle\A\nabla u,\nabla^T\left(\mu^{-1}\A x\right)\nabla u\rangle\,dx
+2\int_{\partial B_r}f\,u\,d\HH,
\end{multline*} 
and the thesis follows thanks to the Lipschitz continuity of $\A$ 
and \eqref{e:mubdd}.
\end{proof}

Let us now deal with the derivative of $\HHH$.

\begin{proposition}\label{p:hprime}
There exists a non negative constant $C_2$ depending on $\lambda$ and on 
the Lipschitz constant of $\A$, such that, for $\cL^1$ 
a.e.~$r\in(0,\mathrm{dist}(\zz,\partial\Om))$,
\begin{equation}\label{e:hprime}
\left|\HHH^\prime(r)-\frac {n-1}r\HHH(r)
-2\int_{\partial B_r}u\langle\A\,\nu,\nabla u\rangle\,d\HH\right|
\leq C_2\HHH(r).
\end{equation} 
\end{proposition}

\begin{proof}
First note that the divergence theorem and the very definition of 
$\mu$ give that
\[
\HHH(r)=\frac 1r\int_{B_r}\div\left(u^2(x)\A(x)x\right)dx,
\]
in turn implying for $\cL^1$ a.e. $r\in(0,\mathrm{dist}(\zz,\partial\Om))$
\begin{align*}
\HHH^\prime(r)&=-\frac 1r \HHH(r)
+\frac 1r\int_{\partial B_r}\div\left(u^2(x)\A(x)x\right)d\HH\\
%=-\frac 1r \HHH(r)+\frac 1r\int_{\partial B_r}
%\left(2u\langle\A(x)x,\nabla u\rangle+u^2\mathrm{Tr}(\nabla\A(x)x)+
%u^2\mathrm{Tr}\A(x)\right)d\HH\\
&=-\frac 1r \HHH(r)+\frac 1r\int_{\partial B_r}
\left(\sum_{i,j=1}^n \frac{\de}{\de x_i}a_{ij}(x)\, x_j 
+\mathrm{Tr}\A\right)u^2d\HH
+2\int_{\partial B_r}u\langle\A\,\nu,\nabla u\rangle\,d\HH.%:=I_1(r)+I_2(r).
\end{align*} 
By (H1) and Lemma~\ref{l:mulip} we get
\[
\frac 1r\left|\int_{\partial B_r}\mathrm{Tr}\A\, u^2d\HH
-n\,\HHH(r)\right|\leq C\int_{\partial B_r}u^2d\HH\leq C\HHH(r),
\]
and
\[
\frac 1r\left|\int_{\partial B_r}\sum_{i,j=1}^n \frac{\de}{\de x_i}a_{ij}(x)\, x_j  u^2(x)d\HH\right|
\leq C\int_{\partial B_r}u^2d\HH\leq C\HHH(r),
\]
from which the conclusion follows.
% \[
% \left|\HHH^\prime(r)-\frac {n-1}r\HHH(r)
% -2\int_{\partial B_r}u\langle\A\,\nu,\nabla u\rangle\,d\HH\right|
% \leq C\HHH(r).\qedhere
% \]
\end{proof}

\subsection{Weiss' monotonicity}
We begin with a Weiss' type quasi-monotonicity formula, that establishes  
the $2$-homogeneity of blow-ups of $u$ in free boundary points. 
\begin{theorem}\label{t:Weiss}
Assume that (H1)-(H3) and \eqref{e:rinormalizzazione} are satisfied. 
There exist nonnegative constants $C_3$, $C_4$ depending on $\lambda$ and 
on the Lipschitz constants of $\A$ and $u$, such that the function 
\[
r\to e^{C_3r}\Phi(r)+C_4\int_0^re^{C_3t}\,t^\alpha dt
\]
is non decreasing on $\big(0,\frac12\mathrm{dist}(\zz,\partial\Om)\wedge 1\big)$. 

More precisely, the following estimate holds true for $\cL^1$-a.e.~$r$ in 
such an interval 
\begin{equation}\label{e:weissmonotonicity}
\frac {d}{dr}\left(e^{C_3r}\Phi(r)+C_4\int_0^re^{C_3t}\,t^{\alpha-1}dt\right)
\geq\frac {2e^{C_3r}}{r^{n+2}}
\int_{\partial B_r}\mu\left(\langle\mu^{-1}\A\,\nu,\nabla u\rangle
-2\,\frac u{r}\right)^2d\HH.
\end{equation}
In particular, the limit $\displaystyle{\Phi(0^+):=\lim_{r\downarrow 0}\Phi(r)}$ 
exists finite.
\end{theorem}
\begin{proof}
By definition for $\cL^1$-a.e. $r\in (0,\mathrm{dist}(\zz,\partial\Om))$ 
we have
\begin{equation}\label{e:phiprime}
\Phi^\prime(r)=\frac {\EEE^\prime(r)}{r^{n+2}}-(n+2)\frac{\EEE(r)}{r^{n+3}}
-2\frac{\HHH^\prime(r)}{r^{n+3}}+2(n+3)\frac {\HHH(r)}{r^{n+4}}.
\end{equation}
First note that \eqref{e:eprime} in Proposition~\ref{p:eprime} yields
\begin{multline*}
\frac {\EEE^\prime(r)}{r^{n+2}}-(n+2)\frac{\EEE(r)}{r^{n+3}}
\geq\frac 2{r^{n+2}}
\int_{\partial B_r}\mu^{-1}\langle\A\,\nu,\nabla u\rangle^2d\HH
+\frac 1{r^{n+3}}\int_{B_r}\langle\A\nabla u,\nabla u\rangle\,
\div\left(\mu^{-1}\A x\right)\,dx\\
-\frac 2{r^{n+3}}\int_{B_r}f\langle \mu^{-1}\A\,x,\nabla u\rangle\,dx
-\frac {C_1}{r^{n+2}}\EEE(r)
-\frac 2{r^{n+3}}\int_{B_r}
\langle\A\nabla u,\nabla^T\left(\mu^{-1}\A x\right)\nabla u\rangle\,dx\\
+\frac 2{r^{n+2}}\int_{\partial B_r}f\,u\,d\HH
%-C_1\frac {\EEE(r)}{r^{n+2}}
-\frac {(n+2)}{r^{n+3}}\int_{B_r}
\langle\A\nabla u,\nabla u\rangle\,dx
-\frac {2(n+2)}{r^{n+3}}\int_{B_r}f\,u\,dx.
\end{multline*}
Then, integrating by part gives
\[
\int_{B_r}\langle\A\nabla u,\nabla u\rangle\,dx+\int_{B_r}f\,u\,dx=
\int_{\partial B_r}u\langle\A\,\nu,\nabla u\rangle\,d\HH.
\]
Thus, we deduce
\begin{multline}\label{e:eprime1}
\frac {\EEE^\prime(r)}{r^{n+2}}-(n+2)\frac{\EEE(r)}{r^{n+3}}\geq
-C_1\frac {\EEE(r)}{r^{n+2}}+\frac 2{r^{n+2}}
\int_{\partial B_r}\mu^{-1}\langle\A\,\nu,\nabla u\rangle^2d\HH\\
+\frac 1{r^{n+3}}\int_{B_r}\left(
\langle\A\nabla u,\nabla u\rangle\,\div\left(\mu^{-1}\A x\right)
-2\langle\A\nabla u,\nabla^T\left(\mu^{-1}\A x\right)\nabla u\rangle
-(n-2)\langle\A\nabla u,\nabla u\rangle\right)\,dx\\
-\frac 2{r^{n+3}}\int_{B_r}f\langle \mu^{-1}\A\,x,\nabla u\rangle\,dx
+\frac 2{r^{n+2}}\int_{\partial B_r}f\,u\,d\HH\\
-\frac{4}{r^{n+3}}\int_{\partial B_r}u\langle\A\,\nu,\nabla u\rangle\,d\HH
-\frac{2n}{r^{n+3}}\int_{B_r}f\,u\,dx.
\end{multline}
Next we employ \eqref{e:hprime} in Proposition~\ref{p:hprime} to infer
\begin{equation}\label{e:hprime1}
-2\frac {\HHH^\prime(r)}{r^{n+3}}+2(n+3)\frac {\HHH(r)}{r^{n+4}}\geq
-2C_2\frac{\HHH(r)}{r^{n+3}}+8\frac {\HHH(r)}{r^{n+4}}
-\frac 4{r^{n+3}}\int_{\partial B_r}u\langle\A\,\nu,\nabla u\rangle\,d\HH.
\end{equation} 
Hence, by taking into account \eqref{e:eprime1} and \eqref{e:hprime1}, 
equation \eqref{e:phiprime} becomes 
\begin{multline}\label{e:phiprime1}
\Phi^\prime(r)+(C_1\vee C_2)\Phi(r)\geq
\frac 2{r^{n+2}}\int_{\partial B_r}\left(
\mu^{-1}\langle \A\,\nu,\nabla u\rangle^2
+4\mu\frac {u^2}{r^{2}}
-4\frac ur\langle\A\,\nu,\nabla u\rangle\right)d\HH
\\
+\frac 1{r^{n+3}}\int_{B_r}\left(
\langle\A\nabla u,\nabla u\rangle\,\div\left(\mu^{-1}\A x\right)
-2\langle\A\nabla u,\nabla^T\left(\mu^{-1}\A x\right)\nabla u\rangle
-(n-2)\langle\A\nabla u,\nabla u\rangle\right)\,dx\\
-\frac 2{r^{n+3}}\left(
\int_{B_r}f\left(\langle \mu^{-1}\A\,x,\nabla u\rangle+n\,u\right)dx
-r\int_{\partial B_r}f\,u\,d\HH\right)=:R_1+R_2+R_3.
%-\frac{2n}{r^{n+3}}\int_{B_r}f\,u
%-2C_2\frac{\HHH(r)}{r^{n+3}}.
%=I_1+\ldots+I_{10}.
\end{multline}
We estimate separately the $R_i$'s.
To begin with, an easy computation shows that 
\begin{equation}\label{e:R1}
R_1=\frac 2{r^{n+2}}
\int_{\partial B_r}\mu\left(\langle\mu^{-1}\A\,\nu,\nabla u\rangle
-2\frac u{r}\right)^2d\HH.
\end{equation}
Moreover, we can rewrite the second term as
\[
R_2=\frac 1{r^{n+3}}\int_{B_r}\left(
\langle\A\nabla u,\nabla u\rangle\,\div\left(\mu^{-1}\A x-x\right)
-2\langle\A\nabla u,\nabla^T\left(\mu^{-1}\A x-x\right)\nabla u\rangle
\right)\,dx.
\]
Then, by the Lipschitz continuity of $\A$ and that of $\mu$ 
in $\zz$, we get
\[
\left\|\nabla\left(\mu^{-1} \A -\Id\right)
\right\|_{L^{\infty}(B_r,\R^{n\times n})}\leq C.
\]
In conclusion, we infer
\begin{equation}\label{e:R2}
|R_2|\leq C\,\frac{\EEE(r)}{r^{n+2}}.
\end{equation}
Finally, we use the identity 
\[
\int_{B_r}
\left(\langle x,\nabla u\rangle+u\,\div x\right)dx=r\int_{\partial B_r}u\,d\HH,
\]
that follows from the Divergence theorem, to rewrite the last term in 
\eqref{e:phiprime1} as
\begin{multline*}
R_3=
-\frac{2f(\zz)}{r^{n+3}}\int_{B_r}\langle \mu^{-1}\A\,x-x,\nabla u\rangle\,dx
\\
-\frac 2{r^{n+3}}\left(\int_{B_r}(f(x)-f(\zz))
\left(\langle \mu^{-1}\A\,x,\nabla u\rangle+n\,u\right)dx
-r\int_{\partial B_r}(f(x)-f(\zz))\,u\,d\HH\right).
\end{multline*}
Hence, by the inequalities in \eqref{e:c11fb}, 
the Lipschitz continuity of $\A$ and that of $\mu$ in $\zz$, 
and the H\"older continuity of $f$ yield, for 
$r\in\big(0,\frac12\mathrm{dist}(\zz,\partial\Om)\wedge 1\big)$,
\begin{equation}\label{e:R3}
|R_3|\leq Cr^{\alpha-1}.
\end{equation}
By collecting \eqref{e:R1}-\eqref{e:R3} we conclude
\begin{equation}\label{e:phiprime2}
\Phi^\prime(r)+C_3\Phi(r)+C_4r^{\alpha-1}\geq
\frac 2{r^{n+2}}
\int_{\partial B_r}\mu\left(\langle\mu^{-1}\A\,\nu,\nabla u\rangle
-2\frac u{r}\right)^2d\HH
\end{equation}
for nonnegative constants $C_3$ and $C_4$.
From this, the Weiss' type monotonicity formula \eqref{e:weissmonotonicity}
follows at once.

Note that the growth estimates in \eqref{e:c11fb} and equalities \eqref{e:E}
and \eqref{e:H} imply that $\Phi(r)$ is bounded for 
$r\in\big(0,\frac12\mathrm{dist}(\zz,\partial\Om)\wedge 1\big)$, so that the existence and 
finiteness of $\Phi(0^+)$ follows directly from \eqref{e:weissmonotonicity}.
\end{proof}

\subsection{Monneau's monotonicity}

Next we prove a Monneau's type quasi-monotonicity formula for singular
free boundary points (cp.~with \cite{Monneau}). We denote by $v$ any positive
$2$-homogeneous polynomial solving 
\begin{equation}\label{e:pdev}
\triangle v=1%f(\zz)
\quad \text{on $\Rn$}.
\end{equation}
Let 
\begin{equation}\label{e:psi}
\Psi_v(r):=
\frac 1{r^{n+2}}\int_{B_r}\left(|\nabla v(x)|^2
%\langle\A(\zz)\nabla v(x),\nabla v(x)\rangle
+2%f(\zz)\,
\,v\right)\,dx-\frac {2}{r^{n+3}}\int_{\partial B_r}v^2d\HH.
\end{equation} 
The expression of $\Psi_v$ is analogous to that of $\Phi$ with coefficients 
frozen in $\zz$ (cp. with \eqref{e:phi} and recall that $\A(\zz)=\Id$ and 
$f(\zz)=\mu(\zz)=1$, by \eqref{e:rinormalizzazione}).
Moreover, since $v$ is $2$-homogeneous and \eqref{e:pdev} holds, we also have
\begin{equation}\label{e:psihom}
\Psi_v(r) \equiv \Psi_v(1) =%f(\zz)\,
\int_{B_1}v\,dx.
\end{equation} 

\begin{theorem}\label{t:Monneau}
Assume (H1)-(H3) and \eqref{e:rinormalizzazione}.
Let $u$ be the minimizer of $\EEE$ on $\KK$ with $\zz\in \Sing(u)$, 
and let $v$ be as above.  
Then, there exists a nonegative constant $C_5$ depending on $\lambda$ 
and on the Lipschitz constant of $\A$, such that 
\[
r\mapsto\int_{\partial B_1}(u_r-v)^2d\HH+C_5(r+r^\alpha)
\]
is nondecreasing on $\big(0,\frac12\mathrm{dist}(\zz,\partial\Om)\wedge 1\big)$.
More precisely, $\cL^1$-a.e. on such an interval
\begin{equation}\label{e:monneaumonotonicity}
\frac d{dr}\left(\int_{\partial B_1}(u_r-v)^2d\HH+C_5\,r^\alpha\right)
\geq\frac 2r(\Phi(r)-\Psi_v(1)).
\end{equation}
\end{theorem}
\begin{proof}
Set $w_r:=u_r-v$. By taking into account equality $\A(\zz)=\Id$
(cp.~with \eqref{e:rinormalizzazione}), the $2$-homogeneity of $v$ and the 
Divergence theorem, we find 
\begin{multline}\label{e:derivative}
\frac d{dr}\int_{\partial B_1}w_r^2d\HH
=\frac 2r\int_{\partial B_1}w_r(\langle\nabla w_r,x\rangle-2w_r)d\HH
=\frac 2r\int_{\partial B_1}w_r(\langle\nabla u_r,x\rangle-2u_r)d\HH\\
=\frac 2r\int_{\partial B_1}w_r(\langle\A(rx)\nabla u_r,x\rangle-2u_r)d\HH+
\frac 2r\int_{\partial B_1}w_r\langle(\A(\zz)-\A(rx))\nabla u_r,x\rangle\,d\HH\\
\geq\frac 2r\int_{\partial B_1}w_r(\langle\A(rx)\nabla u_r,x\rangle-2u_r)d\HH
-2\|\nabla u_r\|_{L^2(\partial B_1)}\|w_r\|_{L^2(\partial B_1)}.
\end{multline}
In view of \eqref{e:c11fb} the latter inequality implies 
\begin{equation}\label{e:m1}
\frac d{dr}\int_{\partial B_1}w_r^2d\HH\geq%+\int_{\partial B_1}w_r^2d\HH\geq
\frac 2r\int_{\partial B_1}w_r(\langle\A(rx)\nabla u_r,x\rangle-2u_r)d\HH-C.
\end{equation}
Next we use an integration by parts, the identity 
\begin{equation}\label{e:pdeur}
\div(\A(rx)\nabla u_r)=f(rx)\chi_{\{u_r>0\}}(x) \quad \text{a.e. and in } \mathcal{D}'(\Omega), 
%\quad \triangle v=f(\zz),
\end{equation}
%(for the second recall that $\zz$ is a singular free boundary point)  
\eqref{e:pdev} and the positivity of $u$ and $v$ to rewrite the first term 
on the right hand side above conveniently 
\begin{multline}\label{e:m2}
\int_{\partial B_1}w_r(\langle\A(rx)\nabla u_r,x\rangle-2u_r)d\HH\\
=\int_{B_1}\left(\langle \A(rx)\nabla u_r,\nabla w_r\rangle
+w_r\,\div(\A(rx)\nabla u_r)\right)\,dx-2\int_{\partial B_1}w_r\,u_r\,d\HH\\
\stackrel{\eqref{e:pdeur}}{=}
\int_{B_1}\left(\langle\A(rx)\nabla u_r,\nabla u_r\rangle+f(rx)u_r\right)dx\\
-\int_{B_1}\left(\langle\A(rx)\nabla u_r,\nabla v\rangle
+v\, f(rx)\chi_{\{u_r>0\}}\right)dx-2\int_{\partial B_1}w_r\,u_r\,d\HH\\
%-2\int_{\partial B_1}(u_r^2-v\,u_r)\,d\HH\\
=\Phi(r)-\int_{B_1}f(rx)(u_r+v\chi_{\{u_r>0\}})dx-\int_{B_1}
\langle\A(rx)\nabla u_r,\nabla v\rangle)\,dx\\
+2\int_{\partial B_1}(\mu(rx)-\mu(\zz))u_r^2d\HH+2\int_{\partial B_1}v\,u_r\,d\HH\\
\stackrel{(\mathrm{H}1),\,\eqref{e:mulip}}{\geq}\Phi(r)-\int_{B_1}f(rx)(u_r+v)dx
-\int_{B_1}\langle\nabla u_r,\nabla v\rangle\,dx
-r\|\A\|_{W^{1,\infty}}\|\nabla u_r\|_{L^2(B_1)}\|\nabla v\|_{L^2(B_1)}\\
-C\|\A\|_{W^{1,\infty}}\,r\int_{\partial B_1}u_r^2d\HH+2\int_{\partial B_1}v\,u_r\,d\HH\\
\stackrel{\eqref{e:scaling},\,\eqref{e:pdev},\,\eqref{e:psihom}}{=}
\Phi(r)-\Psi_v(1)+\int_{B_1}(f(\zz)-f(rx))(u_r+v)dx
+\int_{\partial B_1}u_r(2v-\langle\nabla v,x\rangle)\,d\HH-Cr\\
\stackrel{({\mathrm{H}}3^\prime),\,\text{$v$ $2$-hom}}{\geq}
\Phi(r)-\Psi_v(1)-C\,r^\alpha.
\end{multline}
Thus, by collecting \eqref{e:m1} and \eqref{e:m2} we deduce
\begin{equation}\label{e:m3}
\frac d{dr}\int_{\partial B_1}w_r^2d\HH
\geq\frac 2r(\Phi(r)-\Psi_v(1))-C\,r^{\alpha-1}.
\end{equation}
In conclusion, by \eqref{e:m2} and \eqref{e:m3}, 
\eqref{e:derivative} rewrites as
\[
\frac d{dr}\left(\int_{\partial B_1}w_r^2d\HH+C_5\,r^\alpha \right)
\geq\frac 2r(\Phi(r)-\Psi_v(1)),
\]
for some nonnegative constant $C_5$.
Inequality \eqref{e:monneaumonotonicity} is finally established.
\end{proof}
\begin{remark}
Alternatively, we could establish a slightly different monotonicity 
formula as follows: If in \eqref{e:derivative} we estimate the term
\[
\|\nabla u_r\|_{L^2(\partial B_1)}\|w_r\|_{L^2(\partial B_1)}
\]
by using Cauchy-Schwartz inequality rather than using the boundedness 
in $C^{1}$ of $(u_r)_{r>0}$, we infer
\begin{equation}\label{e:m1.1}
\frac d{dr}\int_{\partial B_1}w_r^2d\HH+\int_{\partial B_1}w_r^2d\HH\geq
\frac 2r\int_{\partial B_1}w_r(\langle\A(rx)\nabla u_r,x\rangle-2u_r)d\HH-C.
\end{equation}
Thus, by collecting \eqref{e:m2} and \eqref{e:m1.1} we deduce
\begin{equation}\label{e:m3.1}
e^{-r}\frac d{dr}\left(e^r\int_{\partial B_1}w_r^2d\HH\right)
%+\int_{\partial B_1}w_r^2d\HH
\geq\frac 2r(\Phi(r)-\Psi_v(1))-C(r^{\alpha-1}+1).
\end{equation}
Finally, from \eqref{e:m3}, \eqref{e:m1.1} and \eqref{e:m3.1} we infer that 
\[
e^{-r}\frac d{dr}\left(e^r\int_{\partial B_1}w_r^2d\HH+\zeta(r)\right)
\geq\frac 2r(\Phi(r)-\Psi_v(1)),
\]
where $\zeta\in C^{0,\alpha}([0,\infty))$ satisfies
\[
\zeta^\prime(r)=C_5e^r\left(r^{\alpha-1}+1\right)\quad\text{on }
\big(0,\frac12\mathrm{dist}(\zz,\partial\Om)\wedge 1\big),
\]
for some nonnegative constant $C_5$. 
\end{remark}

\section{Regularity of the free boundary}\label{s:freeboundary}
Using the quasi-monotonicity formulas above,
in this section we study the regularity of the free boundary
for the obstacle problem for $\EEE$ in \eqref{e:enrg}.
As discussed in the introduction, in view also of recent results by 
Matevosyan and Petrosyan \cite{MaPe}, this approach applies to various
obstacle problems with less regular quasi-linear operators of the type of
certain mean-field models for type II superconductors (cp, e.g., \cite{FeSa10}).

\subsection{Blow-ups}\label{s:bu}
We shall investigate in what follows the existence and uniqueness of the blow-ups.
To this aim, we need to introduce new notation for the rescaled functions 
in any free boundary point similarly to \eqref{e:ur}:
for every point in the free boundary $x_0 \in \Gamma_u$, set
\begin{equation}\label{e:ur bis}
u_{x_0,r}(x):=\frac{u(x_0+rx)}{r^2}.
\end{equation}
%Let us point out that the reference system used in the sequel is the 
%standard cartesian one. 

\begin{remark}\label{r:uniform constants}
A simple corollary of Weiss' quasi monotonicity is the precompactness of the 
family $(u_{x_0,r})_r$ in the topology of $C^{1,\gamma}_{\textup{loc}}(\R^n)$.
Moreover, for base points $x_0$ in a compact set of $\Omega$, the 
$C^{1,\gamma}_{\textup{loc}}(\R^n)$ norms and, thus, the constants in the various
monotonicity formulas \eqref{e:weissmonotonicity}, \eqref{e:monneaumonotonicity}
are uniformly bounded.
Indeed, as pointed out in the corresponding statements, 
they depend on the distance of the point 
from the boundary and the Lipschitz constant of $u$.
\end{remark}

We recall the notation introduced in Section~\ref{s:wm}:
\begin{gather}
\LL(x_0) : = f(x_0)^{-1/2}\A^{1/2}(x_0),\notag\\
u_{\LL(x_0)}(y)= u (x_0+\LL(x_0) y),\notag\\
{\mathbb C}_{x_0}(y)=\A^{-1/2}(x_0)\A(x_0+\LL(x_0)y)\A^{-1/2}(x_0),\notag\\
\intertext{and in addition we set}
u_{\mathbb{L}(x_0),r}(y) := \frac{u (x_0+r\,\mathbb{L}(x_0)\, y)}{r^2},\notag\\
\mu_{\LL(x_0)}(y):=\langle {\mathbb C}_{x_0}(y)\nu(y),\nu(y)\rangle\quad y\neq\zz,\qquad\mu_\LL(\zz):=1,\notag\\
\Phi_{\LL(x_0)}(r):=\EEE_{\LL(x_0)}[u_{\LL(x_0),r},B_1]+
\int_{\de B_1}\mu_{\LL(x_0)}(r y)\,u_{\LL(x_0),r}^2(y)\,d\HH(y).\label{e:phiA}
\end{gather}

In passing we note that $\lambda^{-2}\leq \mu_\LL(y) \leq \lambda^2$ for all $y \in \R^n$.

\begin{proposition}\label{p:blowup}
Let $x_0 \in \Gamma_u$ and $(u_{x_0,r})$ be as in \eqref{e:ur bis}.
Then, for every sequence $r_j \downarrow 0$ there exists a subsequence 
$(r_{j_k})_{k\in\N}\subset (r_j)_{j\in\N}$ such that
$(u_{x_0,r_{j_k}})_{k\in\N}$ converges in $C^{1,\gamma}_{\textup{loc}}(\R^n)$,
for all $\gamma\in (0,1)$, to a function $v(y) = w\big(\LL^{-1}(x_0)y\big)$, 
where $w$ is $2$-homogeneous.
\end{proposition}

\begin{proof}
We drop the dependence on the base point $x_0$ in the subscripts for the sake of
convenience.
Apply to $\Phi_\LL$ the quasi-monotonicity formula in Theorem~\ref{t:Weiss} on
$(r_jr,r_jR)$ for $r\in(0,R)$ and get
\begin{multline}\label{e:monotonia}
e^{C_3r_jR}\Phi_\LL(r_jR)-e^{C_3r_jr}\Phi_\LL(r_jr)
+C_4\int_{r_jr}^{r_jR}e^{C_3t}\,t^{\alpha-1}dt\\
\geq\int_{r_jr}^{r_jR}\frac 2{t^{n+2}}e^{C_3t}
\int_{\partial B_t}\mu_\LL\left(\langle\mu_\LL^{-1}\mathbb{C}\,\nu,\nabla u_{\LL}\rangle
-2\frac{u_{\LL}}{t}\right)^2d\HH\,dt\\
=\int_r^R\frac 2{s^{n+2}}e^{C_3r_js}\int_{\partial B_s}\mu_\LL(r_jy)\left(\langle
\frac{\mathbb{C}(r_jy)\,\nu}{\mu_\LL(r_jy)},\nabla u_{\LL,r_j}\rangle-2u_{\LL,r_j}\right)^2
d\HH\,ds.
\end{multline}
As noticed in Proposition~\ref{p:C11} above, the functions $u_{\LL,r}$ enjoy uniform 
$C^{1,\gamma}_{\textup{loc}}(\R^n)$ estimates, $\gamma\in(0,1)$ arbitrary. 
Therefore, any sequence $(u_{\LL,r_j})_{j\in\N}$ has a convergent subsequence 
in $C^{1,\gamma}_{\textup{loc}}$ to some function $w$, for all $\gamma\in(0,1)$.
Thanks to inequality \eqref{e:monotonia} and recalling that
$\mathbb{C}(\zz)=\Id$ and $\mu_\LL(\zz)=1$, we infer by the 
Lebesgue dominated convergence theorem that $w$ is necessarily 
$2$-homogeneous.
Changing the coordinates back, we conclude as desired.
\end{proof}

\subsection{Quadratic growth}

The following simple generalization of the usual quadratic detachment property 
of the minimizer $u$ from the free boundary holds true.

\begin{lemma}\label{l:quadratic}
There exists a dimensional constant $\theta>0$ such that, for
every $x_0 \in \Gamma_u$ and $r\in(0,\textup{dist}(x_0,\de\Omega)/2)$, it holds
\begin{equation}\label{e:quadratic}
\sup_{x \in \de B_r(x_0)} u(x) \geq \theta\,r^2.
\end{equation}
\end{lemma}

\begin{proof}
First consider a point $y_0 \in N_u$ and $r\in(0,\textup{dist}(y_0, \de \Omega))$,
and define the function
\[
h(x) := u(x) - u(y_0) - \theta|x-y_0|^2,
\]
where $\theta>0$ is a constant to be fixed properly.
Note that $h(y_0) =0$ and that, for some positive constant $C$ 
depending only on $\Om$ and $\|\A\|_{W^{1,\infty}}$, we have
\[
\div (\A \nabla h) = f - 2 \theta \, \div(\A (\cdot-y_0)) 
\stackrel{(H1)\,\&\,(H3)}{\geq} c_0 - 
C\,\theta%(\|\nabla \A\|_{L^\infty}\,\textup{diam}(\Omega) + \|\A\|_{L^\infty})
> 0,
\]
as soon as $\theta>0$ is suitably chosen.
Therefore, by the maximum principle (cp.~\cite{GT}), we deduce that
$\sup_{\de (B_r(y_0) \cap N_u)} h \geq 0$.
Since $h \vert_{B_r(y_0) \cap \Gamma_u} <0$, it follows that $\de B_r(y_0) \cap N_u \neq \emptyset$ and
\[
\sup_{x \in \de B_r(y_0)} u(x) \geq \theta\,r^2.
\]
Since the radius does not depend on $y_0$ and the supremum on $\de B_r(y_0)$
is continuous with respect to $y_0$,
applying this reasoning to a sequence $y_k\in N_u$ converging to $x_0$, we conclude
\eqref{e:quadratic}.
\end{proof}

\subsection{Classification of blow-ups}

As a simple corollary of Proposition~\ref{p:blowup} and Lemma~\ref{l:quadratic}, 
we infer that if $w$ is a $2$-homogeneous limit of a converging 
sequence of rescalings $(u_{x_0,r_j})_{j\in\N}$, in a free boundary point 
$x_0 \in \Gamma_u$, then $\zz \in \Gamma_w$, i.e.~$w\not\equiv 0$ in any 
neighborhood of $\zz$. We show next some other properties of such limits $w$.
To this aim we recall first the results established in the classical case.

A \textit{global solution} to the obstacle problem is a positive function
$w \in C_{\textup{loc}}^{1,1}(\R^n)$ solving \eqref{e:pdeu} with $\A\equiv \Id$
and $f\equiv 1$.
The following theorem is due to Caffarelli \cite{Ca77,Ca98}.

\begin{theorem}\label{t:convex}
Every global solution $w$ is convex.
Moreover, if $w$ is non-zero and homogeneous of degree $2$, then one
of the following two cases occur:
\begin{itemize}
\item[(A)] $w(y) = \frac{1}2 \,\big(\langle y, \nu \rangle \vee 0\big)^2$ for some
$\nu \in \mathbb{S}^{n-1}$;
\item[(B)] $w(y) = \langle \mathbb{B}\, y,y \rangle $ with $\mathbb{B}$ a symmetric, positive definite matrix satisfying $\textup{Tr}(\mathbb{B}) = \frac{1}2$.
\end{itemize}
\end{theorem}

Having this result at hand, a complete classification of the blow-up limits
for the obstacle problem for $\EEE$ follows as in the classical setting.

\begin{proposition}[Classification of blow-ups]\label{p:classification}
Every blow-up $v_{x_0}$ at a free boundary point $x_0$ is of the form 
$v_{x_0}(y)=w(\LL^{-1}(x_0)y)$ with $w$ a non-trivial, 
$2$-homogeneous global solution.
\end{proposition}

\begin{proof}
We use the notation at the beginning of Section~\ref{s:bu},
dropping the dependence on $x_0$ in the subscripts. % for the sake of convenience. 
Denote by $w$ the limit in the $C^{1,\gamma}_{\textup{loc}}$ topology 
of $(u_{\LL,r_j})_{j\in\N}$, for some $r_j\downarrow 0$; and consider 
the energies defined on $H^1(B_1)$ by
\[
\mathscr{F}_j(v) := \int_{B_1} \Big(
\langle \mathbb C (r_jy) \nabla v(y), \nabla v(y) \rangle 
+ 2\frac{f_\LL(r_jy)}{f(x_0)} \, v(y)\Big) dy,
\]
if $v\geq 0$ $\cL^n$ a.e. on $B_1$ and $v=u_{\LL,r_j}$ on $\de B_1$, $\infty$ 
otherwise. By definition, the rescaled function $u_{\LL,r_j}$ itself is the minimizer 
of $\mathscr{F}_j$. Recalling that $\mathbb C (\zz) = \Id$ and $f_\LL(\zz)=f(x_0)$, 
it follows easily that $(\mathscr{F}_j)_{j\in\N}$ $\Gamma$-converges with respect 
to the strong $H^1$ topology to
\[
\mathscr{F}(v) := \int_{B_1} \big(|\nabla v|^2 + 2 \, v\big) dx,
\]
if $v\geq 0$ $\cL^n$-a.e. on $B_1$ and $v=w$ on $\de B_1$, $\infty$ otherwise on 
$H^1(B_1)$. Therefore, according to Proposition~\ref{p:blowup}, we infer that $w$ 
is a $2$-homogeneous function minimizing $\mathscr{F}$ on $B_1$. That is, extending 
$w$ by $2$-homogeneity to $\R^n$, $w$ is a non-trivial, $2$-homogeneous global 
solution. In conclusion, as $u_{x_0,r_j}(\LL^{-1}(x_0)y)=u_{\LL,r_j}(y)$,
we infer that $u_{x_0,r_j} \to v= w (\LL^{-1}(x_0) y)$ in 
$C^{1,\gamma}_{\textup{loc}}$.
\end{proof}

The above proposition allows us to formulate a simple criterion to distinguish between \textit{regular} and \textit{singular} free boundary points.

\begin{definition}\label{d:reg sing}
A point $x_0\in \Gamma_u$ is a \textit{regular} free boundary point, and
we write $x_0 \in \Reg(u)$, if there exist a blow-up of $u$ at $x_0$ of type (A).
Otherwise, we say that $x_0$ is \textit{singular}, and write $x_0 \in \Sing(u)$.
\end{definition}

Simple calculations show that $\Psi_w(1) = \vartheta$ for every global solution 
of type (A) and $\Psi_w(1) = 2\,\vartheta$ for every global solution of type (B),
where $\Psi_w$ is the energy defined in \eqref{e:psi} and $\vartheta$ is a 
dimensional constant. Therefore, by Weiss' quasi monotonicity it follows that
a point $x_0 \in \Gamma_u$ is regular if and only if $\Phi_{\LL(x_0)}(0) = \vartheta$,
% \[
% \Phi_{x_0}(0) = % f(x_0) \det\A^{-1/2}(x_0) \, 
% \vartheta,
% \]
or, equivalently, if and only if
\textit{every} blow-up at $x_0$ is of type (A).

\subsection{Uniqueness of blow-ups}
The last remarks show that the blow-up limits at the free boundary points are of a 
unique type: at a given point they are always either of type (A) or of type 
(B). Nevertheless, this does not imply the uniqueness of the limiting profile
independently of the converging sequence $r_j\downarrow0$. 
We show next that this is the case.

In the classical framework, the uniqueness of the blow-ups can be derived 
\textit{a posteriori} from the regularity properties of the free boundary
established thanks to an argument by Caffarelli employing cones of monotonicity.
Those are, in turn, obtained via a PDE argument for the gradient of the solution
$u$. In our case, due to the lack of regularity of the matrix of
the coefficients $\A$, we need 
to prove it \textit{a priori}, following the approaches by Weiss and Monneau.

For regular points, we need to introduce the following deep result by Weiss 
\cite[Theorem~1]{Weiss}. For ease of readability we recall the notation 
introduced in \eqref{e:psi} for $v$ any positive 2-homogeneous polynomial:
\[
\Psi_v(1)= \int_{B_1}\big(|\nabla v|^2 + 2\,v \big)\, dx - 2\int_{\de B_1}v^2 d\HH.
\]

\begin{theorem}[Weiss' epiperimetric inequality]\label{t:epiperimetrica}
There exist dimensional constants $\delta,\kappa>0$ with this property:
for every $2$-homogeneous function $\varphi \in H^1(B_1)$ with
$\|\varphi - w \|_{H^1(B_1)} \leq \delta$
for some global solution $w$ of type (A),
there exists $\zeta \in H^1(B_1)$ such
that $\zeta\vert_{\de B_1} = \varphi\vert_{\de B_1}$ and
\begin{equation}\label{e:epi}
\Psi_\zeta(1) - \vartheta \leq (1- \kappa) \big(\Psi_\varphi(1) - \vartheta\big).
\end{equation}
\end{theorem}

We now proceed with the proof of the uniqueness of the blow-ups at regular points.
A preliminary step in this direction is the following lemma.

\begin{lemma}\label{l:epi vale}
Let $u$ be a solution of the obstacle problem with $\zz \in \Gamma_u$ and
assume that \eqref{e:rinormalizzazione} holds.
If there exist radii $0\leq s_0 < r_0 <1$ such that
\begin{equation}\label{e:vicino}
\inf_{w}\| u_r\vert_{\de B_1} - w \|_{H^1(\de B_1)} \leq \delta \quad 
\text{for all }\; s_0\leq r \leq r_0,
\end{equation}
where the infimum is taken among all global solutions $w$ of type (A) and $\delta>0$ 
is the constant in Theorem~\ref{t:epiperimetrica}, then 
for every $s_0\leq s\leq t \leq r_0$ we have
% the following holds
%\begin{equation}\label{e:decay}
%\Phi(t) - \vartheta \leq C_7\,t^{C_6},
%\end{equation}
%where $C_6,C_7>0$ is a dimensional constant.
\begin{align}\label{e:diff}
\int_{\de B_1} \left |u_t -u_s\right| d\HH
\leq{}&  C_7\, t^{C_6},
\end{align}
where $C_6, C_7>0$ are constants depending on the Lipschitz constants
of $\A$ and $u$.
\end{lemma}

\begin{proof} 
By means of Remark \ref{r:Ograndi} we can compute the derivative of $\Phi(r)$ 
in the following way:
\begin{align*}%\label{e:Phi' nuova}
\Phi'(r) & = \frac {\EEE^\prime(r)}{r^{n+2}}-(n+2)\,\frac{\EEE(r)}{r^{n+3}}
-2\,\frac{\HHH^\prime(r)}{r^{n+3}}+2\,(n+3)\,\frac {\HHH(r)}{r^{n+4}}\notag\\
& \stackrel{\mathclap{\eqref{e:hprime1},\,\eqref{e:scaling}}}{\geq}\quad \frac{1}{r^{n+2}} \int_{\de B_r} \big(\langle \A\,\nabla u, \nabla u\rangle + 2\,f\,u \big)
-(n+2)\,\frac{\EEE(r)}{r^{n+3}} 
+ 8\frac{\HHH(r)}{r^{n+4}}  \notag\\
&\quad - \frac{4}{r^{n+3}} \int_{\de B_r} u\, \langle \A\,\nu, \nabla u \rangle - C \notag\\
& \stackrel{\eqref{e:scaling}}{\geq} \quad
\frac{1}{r^{n+2}} \int_{\de B_r} \big(|\nabla u|^2 + 2\,u \big) 
- \frac{n+2}{r} \, \Phi(r)
- \frac{2}{r^{n+4}}(n-2) \int_{\de B_r} u^2 d\HH \notag\\
&\quad - \frac{4}{r^{n+3}} \int_{\de B_r} u\, \langle \nu, \nabla u \rangle 
- C\,r^{\alpha-1}\notag\\
& = - \frac{n+2}{r} \, \Phi(r) + \frac{1}{r} \int_{\de B_1}
\Big(
\big(\langle \nu,\nabla u_r\rangle - 2\,u_r\big)^2 + |\de_\tau u_r|^2+2\,u_r -2\,n\,u_r^2
 \Big)d\HH- C \,r^{\alpha -1} ,
\end{align*}
where we denoted by $\de_\tau u_r$ the tangential derivative of $u_r$ along $\de B_1$.
Let $w_r$ be the $2$-homogeneous extension of $u_r|_{\partial B_1}$, then a simple 
integration in polar coordinates gives 
\begin{align*}
\int_{\de B_1}
\big( |\de_\tau u_r|^2+2\,u_r -2\,n\,u_r^2  \big)d\HH & = 
\int_{\de B_1}\big( |\de_\tau w_r|^2+2\,w_r + 4\, w_r^2-2\,(n+2)\,w_r^2  \big)d\HH\\
& = (n+2)\,\Psi_{w_r}(1).
\end{align*}
Therefore, we conclude that
\begin{equation}\label{e:Phi' nuovo bis}
\Phi'(r) \geq \frac{n+2}{r} \,\big(\Psi_{w_r}(1) - \Phi(r) \big) 
+ \frac{1}{r} \int_{\de B_1}
\big(\langle \nu,\nabla u_r\rangle - 2\,u_r\big)^2d\HH - C \,r^{\alpha -1}.
\end{equation}
By \eqref{e:vicino} we can apply the epiperimetric inequality \eqref{e:epi} to $w_r$,
and find a function $\zeta\in H^1(B_1)$ with $\zeta\vert_{\de B_1} = u_r\vert_{\de B_1}$
such that
\begin{equation}\label{e:epi applicata}
\Psi_\zeta(1) - \vartheta \leq
(1-\kappa) \big(\Psi_{w_r}(1) - \vartheta\big).
\end{equation}
Moreover, we can assume without loss of generality (otherwise we substitute $\zeta$ with $u_r$ itself) that $\Psi_\zeta(1) \leq \Psi_{u_r}(1)$.
Note that, by freezing the coefficients as usual, hypothesis (H1)-(H3) and the minimality
of $u_r$ for the energy $\EEE$ with respect to its boundary conditions, we have that
\begin{align}\label{e:freezing zeta}
\Psi_\zeta(1) & = \int_{B_1}\big( |\nabla \zeta|^2 +2\,\zeta \big)\,dx - 2\int_{\de B_1} \zeta^2d\HH\notag\\
& \geq \int_{B_1}\big( \langle \A(rx)\,\nabla \zeta,\nabla \zeta\rangle^2 +2\,f(rx)\,\zeta\big)\,dx - 2\int_{\de B_1} \mu(rx)\,\zeta^2 d\HH\notag\\
& \quad - C\,r^\alpha\,\int_{B_1}\big( |\nabla \zeta|^2 +2\,\zeta\big)\,dx
- C\,r\int_{\de B_1} \zeta^2
\notag\\
& \geq \Phi(r) - C\,r^\alpha\, \int_{B_1}\big( |\nabla u_r|^2 +2\,u_r\big)\,dx  - C\,r\int_{\de B_1} u_r^2\notag\\
& \geq \Phi(r) - C\,r^\alpha.
\end{align}
Combining together \eqref{e:epi applicata} and \eqref{e:freezing zeta}, 
we finally infer that
\begin{equation}\label{e:guadagno}
\Psi_{w_r}(1) - \Phi(r) \geq \frac{1}{1-\kappa}\big(\Phi(r) - \vartheta - C\,r^\alpha \big) + \vartheta - \Phi(r) = \frac{\kappa}{1-\kappa} \big(\Phi(r) - \vartheta \big)
- C\, r^\alpha.
\end{equation}
Therefore, we can conclude from \eqref{e:Phi' nuovo bis} and \eqref{e:guadagno} that
\begin{equation}\label{e:diff ineq}
\Phi'(r) \geq \frac{n+2}{r}\, \frac{\kappa}{1-\kappa} \big(\Phi(r) - \vartheta \big)
- C\, r^{\alpha-1}.
\end{equation}
Let now $C_6$ be any exponent in $(0,\alpha\wedge(n+2)\frac{\kappa}{1-\kappa})$,
%Setting for simplicity $C_6 = (n+2)\frac{\kappa}{1-\kappa}$, and assuming without 
%loss of generality that $C_6 < \alpha$, this implies that
then
\begin{equation}\label{e:g}
\left(\big(\Phi(r) - \vartheta\big) \, r^{-C_6}\right)' \geq - C\, r^{\alpha-1-C_6},
\end{equation}
and by integrating in $(t,r_0)$ for $t\geq s_0$, we finally get from \eqref{e:C11}
\[
\Phi(t) -\vartheta \leq C\,\left(t^{C_6} + t^\alpha\right) \leq C_7\,t^{C_6}.
\]
Consider now $s_0 < s < t <r_0$ and estimate as follows
\begin{align}
\int_{\de B_1} \left |u_t -u_s\right| d\HH \leq{}&
\int_s^t \int_{\de B_1} r^{-2}\left |\langle \nabla u (rx), \nu(x) \rangle 
- 2 \frac{u(rx)}{r}\right| d\HH(x) \notag\\
={}&\int_s^t r^{-1} \int_{\de B_1} \left |\langle\nabla u_r,\nu\rangle-2u_r\right| 
d\HH\,dr\notag\\
\leq{}& (n\,\omega_n)^{1/2}\int_s^t r^{-1/2} \left( r^{-1}\int_{\de B_1} 
\left(\langle \nabla u_r, \nu \rangle - 2 u_r\right)^2 d\HH\right)^{1/2} dr.\notag
\end{align}
Combining \eqref{e:Phi' nuovo bis}, \eqref{e:guadagno} and H\"older inequality,
we then have
\begin{align*}%\label{e:diff}
\int_{\de B_1} \left |u_t -u_s\right| d\HH
\leq{}& C
\int_s^t r^{-1/2} \left( \Phi'(r) +C\,r^{\alpha -1 }\right)^{1/2} dr
\notag\\
\leq{}& C\, \left(\log \frac ts\right)^\frac12 \big( \Phi(t) - \Phi(s) 
+C\,(t^\alpha -s^\alpha) \big)^{1/2}
\leq C\, \left(\log \frac ts\right)^{1/2}\, t^{\frac{C_6}{2}}.
\end{align*}
A simple dyadic decomposition argument then leads to the conclusion. 
Indeed, if $s \in [2^{-k}, 2^{-k+1})$ and $t \in [2^{-h}, 2^{-h+1})$ with 
$h \leq k$, applying the estimate above iteratively on dyadic intervals, 
we infer for $q = 2^{\frac{C_6}{2}}$ and a dimensional constant $C>0$,
\begin{align*}
\int_{\de B_1} \left |u_t -u_s\right| d\HH
\leq{}& C \sum_{j=h}^k q^{-j} \leq C\, q^{-h} \leq C\, t^{\frac{C_6}2}.\qedhere
\end{align*}
\end{proof}
\begin{remark}\label{r:monotonicity}
Formula \eqref{e:Phi' nuovo bis} yields Weiss' quasi-mononicity discarding
both Payne-Weinberger's formula and the PDE solved by $u$. 
Indeed, by taking into account the minimality of $u_r$ with respect to its 
boundary datum, directly from \eqref{e:Phi' nuovo bis} we infer that
\[
\Phi^\prime(r)\geq
\frac 1r\int_{\partial B_1}(\langle\nu,\nabla u_r\rangle-2u_r)^2d\HH
-C\,r^{\alpha-1},
\]
in turn implying
\[
\Phi(r) + C\, r^\alpha - \Phi(s) - C\, s^\alpha \geq \int_s^r \frac1t
\int_{\partial B_1}(\langle\nu,\nabla u_t\rangle-2u_t)^2d\HH\,dt. 
\]
\end{remark}

We can now prove the uniqueness of the blow-ups at regular points of the free boundary.

\begin{proposition}\label{p:uniqueness blowup reg}
Let $u$ be a solution to the obstacle problem \eqref{e:pdeu} and
$x_0\in \Reg(u)$. Then, there exist constants $r_0=r_0(x_0), \, \eta_0=\eta_0(x_0) >0$
such that every $x \in \Gamma_u \cap B_{\eta_0}(x_0)$ is a regular point and,
denoting by $v_{x}(y)= w (\LL^{-1}(x) y )$
any blow-up of $u$ at $x$, we have
\begin{equation}\label{e:decay unico}
\int_{\de B_1} \left |u_{\LL(x), r} -w\right| d\HH(y)
\leq C\, r^{\frac{C_6}2} \quad \textrm{for all }\; r\in(0,r_0),
\end{equation}
where $C$ and $\gamma>0$ are dimensional constants. 
In particular, the blow-up limit $v_x$ is unique.
\end{proposition}

\begin{proof}
Denote by $\Phi(x,r)$ the boundary adjusted energy \eqref{e:phi} with base
point $x$, i.e. with domain of integration $B_r(x)$ rather than $B_r$.
The upper semicontinuity of $\Gamma_u \ni x\mapsto \Phi(x,0^+)$ follows
from Weiss' quasi-monotonicity, that in turn yields that $\Reg(u)\subset\Gamma_u$ 
is relatively open, thus proving the first claim if $\eta_0$ is sufficiently small.

By Proposition~\ref{p:C11}, given $\bar \eta>0$
such that $B_{\bar\eta}(x_0) \subset\subset \Omega$ and
$\Gamma_u \cap B_{\bar\eta}(x_0) = \Reg(u)$, then
\[
C_8 := \sup_{x\in \Gamma_u \cap B_{\bar \eta}(x_0), r <\bar \eta} 
\|u_{\LL(x),r}\|_{C^{1,\gamma}(\de B_1)} < \infty.
\]
Let $\delta>0$ be the constant in Theorem~\ref{t:epiperimetrica}.
By compactness, if $\|g\|_{C^{1,\gamma}(\de B_1)}\leq C_8$, we can find
$\eps >0$ such that 
\begin{equation}\label{e:epsdelta}
\|g\|_{L^1(\de B_1)}\leq \eps \quad \Longrightarrow \quad
\|g\|_{H^1(\de B_1)} \leq \frac\delta4.
\end{equation}
Next, we fix $\bar r_0>0$ such that $C_7\, \bar r_0^{C_6} \leq \eps$ and
\begin{equation}\label{e:vicino8}
\inf_{w}\| u_{\LL(x_0),\bar r_0}\vert_{\de B_1} - w \|_{H^1(\de B_1)} \leq \frac\delta4,
\end{equation}
where the infimum is taken among all global solutions $w$ of type (A). 
To show the existence of such a treshold $\bar r_0$, we argue by contradiction:
if it does not exist, we must find a sequence $r_j$ converging to $0$ such that
$\| u_{\LL(x_0),r_j} - w \|_{H^1(\de B_1)} \geq \delta$ for every
$w$ global solution of type (A).
On the other hand, since $x_0$ is a regular free boundary point, 
up to subsequences, not relabeled for conveniene, $(u_{\LL(x_0), r_j})_{j\in\N}$ 
converges in $C^{1,\gamma}_{\textup{loc}}$ to a blow-up $v$ of $u$ at $x_0$ of type 
(A), thus giving a contradiction.

By the continuity of $\A$ and $f$, there exists $0<\eta_0 \leq \bar \eta$ such that
for all $x\in \Reg(u)\cap B_{\eta_0}(x_0)$,
\begin{equation}\label{e:vicino9}
\inf_{w}\| u_{\LL(x),\bar r_0}\vert_{\de B_1} - w \|_{H^1(\de B_1)} \leq \frac\delta2,
\end{equation}
where the infimum is considered in the same class of functions as above.
We claim that in turn this implies that for all $x\in \Reg(u) \cap B_{\eta_0}(x_0)$
and $0< r \leq \bar{r}_0$
\begin{equation}\label{e:vicino10}
\inf_{w}\| u_{\LL(x),r}\vert_{\de B_1} - w \|_{H^1(\de B_1)} \leq \delta.
\end{equation}
To this aim, fix $x \in \Reg(u) \cap B_{\eta_0}(x_0)$ and let $s_0< \bar r_0$ be the
maximal radius such that \eqref{e:vicino10} holds for every $s_0\leq r\leq \bar r_0$.
Assume that $s_0>0$ and note that, in particular,
\begin{equation}\label{e:vicino11}
\inf_{w}\| u_{\LL(x),s_0}|_{\partial B_1} - w \|_{H^1(\de B_1)} = \delta.
\end{equation}
Then, by Lemma~\ref{l:epi vale}
(recall that, being $B_{\eta_0}(x_0) \subset\subset \Omega$,
the constants are uniform at points in $\Gamma_u \cap B_{\eta_0}(x_0)$ -- cp.~Remark~\ref{r:uniform constants}),
we infer that $\|u_{\LL(x),s}-u_{\LL(x),t}\|_{L^1(\de B_1)}\leq C_7\,\bar r_0^{C_6}$ for every $s,t \in [s_0, \bar r_0]$.
Since the functions $u_{\LL(x),s}$ are equibounded in $C^{1,\gamma}(\de B_1)$ by
$C_8$, \eqref{e:epsdelta} gives that
\[
\|u_{\LL(x),s}-u_{\LL(x),t}\|_{H^1(\de B_1)}\leq \frac\delta4 \quad \text{for every} \quad s,t \in [s_0, \bar r_0].  
\]
In particular, by \eqref{e:vicino9} and the
triangular inequality, we get a contradiction to \eqref{e:vicino11}.

We are now ready for the conclusion. Thanks to \eqref{e:vicino10}, we deduce that
\eqref{e:diff} in Lemma~\ref{l:epi vale} holds for every $s,t\in(0, \bar r_0)$.
Therefore, by passing to the limit as $s\downarrow 0$ in \eqref{e:diff} we find
\[
\int_{\de B_1} \left |u_{\LL(x), t} -w\right| d\HH
\leq C\, t^{\frac{C_6}2},
\]
and thus the uniqueness of the blow-up limit is established.
\end{proof}

To prove uniqueness of blow-ups for singular point we need to establish the 
counterpart of Lemma~\ref{l:epi vale} in this setting, though we do not get
a rate for the convergence of the rescalings to their blow-up limits.

\begin{proposition}\label{p:uniform continuity}
For every point $x$ of the singular set $\Sing(u)$
there exists a unique blow-up limit $v_x(y) = w\big(\LL^{-1}(x) y)$.
Moreover, if $K$ is a compact subset of $\Sing(u)$, then, for every point 
$x\in K$,
\begin{equation}\label{e:decay unico sing}
\|u_{\LL(x),r} -w \|_{C^1(B_1)} \leq \sigma_K(r) \quad 
\textrm{for all } r\in(0,r_K),
\end{equation}
for some modulus of continuity $\sigma_K:\R^+ \to \R^+$ and a radius 
$r_K>0$. 
\end{proposition}

\begin{proof}
With no loss of generality we show the uniqueness property in case 
the base point $x\in\Sing(u)$ is actually the origin $\zz$ and
\eqref{e:rinormalizzazione} holds.
We use Monneau's quasi monotonicity formula in Theorem~\ref{t:Monneau}.
To this aim, we suppose that $(u_{r_j})_{j\in\N}$ converges in 
$C^{1,\gamma}_{\textrm{loc}}$, $\gamma\in(0,1)$ arbitrary, to a 
$2$-homogeneous quadratic polynomial $v$ with $\mathrm{Tr}(D^2v) = 1$.
Then, from \eqref{e:monneaumonotonicity} we infer that
\[
\lim_j\int_{\partial B_1}(u_{r_j}-v)^2d\HH=0.
\]
In turn, this implies that the monotone function
\[
r\to\int_{\partial B_1}(u_r-v)^2d\HH+C_5(r+r^\alpha)
\]
is infinitesimal as $r\downarrow 0$. In particular, for all infinitesimal 
sequences $h_j$ we have that $(u_{h_j})_{j\in\N}$ converges to $v$ in 
$C^{1,\gamma}_{\textrm{loc}}$, the uniqueness of the limit then follows at 
once.

Having fixed a compact subset $K$ of $\Sing(u)$, to prove the uniform convergence 
we argue by contradiction. Assume there exist points $x_j\in K$ %,with $x_j\to z\in K$, 
and radii $r_j\downarrow 0$ for which the rescalings 
$u_{\LL(x_j),r_j}$ and the blow-ups $v_{x_j}$ of $u$ at $x_j$ satisfy
\[
\|u_{\LL(x_j),r_j} - v_{x_j} \|_{C^1(B_1)} \geq \eps >0,\quad\text{for some $\eps$}.
\]
Thanks to Proposition~\ref{p:C11} we may assume that, up to subsequences not 
relabeled, $(u_{\LL(x_j),r_j})_{j\in\N}$ converges in $C^{1,\gamma}_{\mathrm{loc}}$ 
to a function $w$. Moreover, by taking into account that the constants in Weiss' 
quasi-monotonicity are bounded since the points are varying on a compact set, 
we may argue as in Propositions~\ref{p:blowup} and \ref{p:classification} to 
deduce that the limit $w$ is actually a $2$-homogeneous global solution
(cp.~\eqref{e:monotonia}). 

Let $\Phi_{\LL(x_j)}$ be as in \eqref{e:phiA}.
It is elementary to show that $\Phi_{\LL(x_j)}(r)\to\Psi_w(r)$ for all $r>0$. 

Then, using Lemma~\ref{l:quadratic} and the classification of free boundary points 
according to the energy, we get $\zz \in \Sing(w)$. 
Indeed, if not, as $\vartheta=\Psi_w(0^+)=\Psi_w(r)$ for all $r$, we would infer 
that $\Phi_{\LL(x_j)}(\rho)\leq\frac32\vartheta$ for $j$ big enough for some fixed 
$\rho>0$. In turn, the latter condition is a contradiction to 
$\Phi_{\LL(x_j)}(0^+)=2\,\vartheta$ that follows from the quasi-monotonicity of 
$\Phi_{\LL(x_j)}$ as $x_j\in\Sing(u)$.

We claim next that $w(y) = \langle \mathbb{B}\,y, y\rangle$, for some positive, 
symmetric $\mathbb{B}$ with $\textup{Tr}(\mathbb{B}) =\frac12$, i.e.~$w$ coincides 
with its blow-up in $\zz$.
To prove this, note that, $\Lambda_{w}$ is a convex set by Theorem~\ref{t:convex}, 
and thus it is a cone since $\zz \in \Lambda_w$.
Therefore, $\cL^n(\Lambda_{w}) = \cL^n(\Lambda_{v_{\zz}})$, where $v_{\zz}$ is the 
blow-up of $w$ at $\zz$. As $\zz \in \Sing(w)$, the latter equality implies that 
$\cL^n(\Lambda_{w})=0$.
Hence, by equation \eqref{e:pdev} and Lioville's theorem, $w$ is a 
$2$-homogeneous polynomial.

In conclusion, by taking this into account and the fact that all norms are
equivalent for polynomials, Monneau's quasi monotonicity formula provides a 
contradiction (note that the constants therein are bounded since the points 
are varying in a compact set -- cp.~Remark~\ref{r:uniform constants}):
\begin{align*}
0<\eps \leq & \limsup_{j\to +\infty} \|u_{\LL(x_j),r_j}-v_{x_j}\|_{C^1(B_1)}\leq
\limsup_{j\to +\infty}\|w-v_{x_j}\|_{C^1(B_1)}\\
&\leq C\, \limsup_{j\to +\infty} \|w-v_{x_j}\|_{L^2(\de B_1)}
\stackrel{\eqref{e:monneaumonotonicity}}{\leq} C\, 
\limsup_{j\to +\infty} \|w-u_{\LL(x_j),r_j}\|_{L^2(\de B_1)} =0.\qedhere
\end{align*}
\end{proof}

\subsection{Regular free boundary points}
We are now ready to establish the $C^{1,\beta}$ regularity of the free boundary 
in a neighborhood of any point $x$ of $\Reg(u)$. 
Recall that blow-up limits in regular points are unique
(cp.~Proposition~\ref{p:uniqueness blowup reg}), so that denoting by 
$n(x)\in\mathbb S^{n-1}$ the blow-up direction at $x\in \Reg(u)$, we have
\[
v_{x}(y)=\frac12(\langle \LL^{-1}(x) n(x), y \rangle \vee 0)^2.
\]
As usual, we shall state and prove the result below with base point the origin. 
We follow here the arguments in \cite{Weiss}.

\begin{theorem}\label{t:regular}
Let $\zz \in \Reg(u)$. 
Then, there exists $r>0$ such that $\Gamma_u \cap B_r$ is a $C^{1,\beta}$ 
hypersurface for some universal exponent $\beta\in (0,1)$.
\end{theorem}

\begin{proof}
Let $\eta_0=\eta_0(\zz)$ and $r_0=r_0(\zz)$ be the radii provided by 
Proposition \ref{p:uniqueness blowup reg}.
We start off showing that for a universal constant $C>0$ and a universal 
(computable) exponent $\beta\in (0,1)$  
\begin{equation}\label{e:holder normal}
|\LL^{-1}(x)\,n(x) - \LL^{-1}(z)\,n(z)| 
\leq C\, |x-z|^\beta,
\end{equation}
for every $x$ and $z\in \Reg(u)\cap B_{{\eta_0}/2}$. 
To this aim, let $s\in(0,r_0)$, then
\begin{align}\label{e:cont blow-up}
\|v_x - v_z\|_{L^1(\de B_1)}
&\leq \|v_x - u_{\LL(x),s}\|_{L^1(\de B_1)} + 
\|u_{\LL(x),s} - u_{\LL(z),s}\|_{L^1(\de B_1)} +
\|u_{\LL(z),s} - v_z\|_{L^1(\de B_1)}\notag\\
&\stackrel{\mathclap{\eqref{e:decay unico}}}{\leq} C\, s^{\frac{C_6}2} + 
\|u_{\LL(x),s} - u_{\LL(z),s}\|_{L^1(\de B_1)}.
\end{align}
By taking into account that the map $y\to \LL(y)$ is 
H\"older continuous with exponent $\theta:=\alpha\wedge 1/2$ thanks to 
(H1) and (H3), in view of estimate \eqref{e:c11fb} we can bound the 
second term above as follows
\begin{align}\label{e:cont blow-up1}
&\|u_{\LL(x),s} - u_{\LL(z),s}\|_{L^1(\de B_1)}\notag\\
&\leq \int_{\de B_1}\int_0^1 s^{-2}\Big( 
|\nabla u (t(z+s\, \LL(z)y)+(1-t)(x+s\,\LL(x)y)|
% \Big.\cdot
% \notag\\
% &\hskip1cm\Big.\cdot
|z-x+s(\LL(z)-\LL(x))| dt\Big) 
\,d\HH(y)\notag\\
&
%\hskip1cm
\leq C\,s^{-2}
(|z-x|+s+s|z-x|^\theta)\cdot(|z-x|+s|z-x|^{\theta})\leq C\, |z-x|^{\theta},
\end{align}
if $s=|z-x|^{1-\theta}$, and 
$C=C(n,\|\LL\|_{L^\infty(B_{\eta_0/2},\R^{n\times n})})$.
Therefore, if $\beta:=\theta\wedge\frac{C_6}2(1-\theta)$,
\eqref{e:holder normal} follows from \eqref{e:cont blow-up},
\eqref{e:cont blow-up1} and the simple observation that for some dimensional 
constant $C>0$ it holds   
\begin{equation*}
|\LL^{-1}(x)\,n(x) - \LL^{-1}(z)\,n(z)| \leq
C\|v_x - v_z\|_{L^1(\de B_1)}, %\int_{\de B_1} |v_x - v_y|d\HH,
\end{equation*}
as the right hand side above is a norm on $\R^n$.

Next, consider the cones $C^\pm(x,\eps)$, $x \in \Reg(u)$, given by
\[
C^\pm(x, \eps) := \left\{ y\in \R^n \,:\, \pm\langle y-x, 
\frac{\A^{-1/2}(x)n(x)}{|\A^{-1/2}(x)n(x)|}\rangle \geq \eps |y-x| \right\}.
\]
We claim that, for every $\eps>0$, there exists $\delta>0$ such
that, for every $x \in \Reg(u) \cap B_{\eta_0/2}$, 
\begin{equation}\label{e:claimconi}
C^+(x, \eps) \cap B_\delta(x) \subset N_u \;\text{ and }\;
C^-(x, \eps) \cap B_\delta(x) \subset \Lambda_u. 
\end{equation}
For, assume by contradiction that there exist $x_j \in \Reg(u) \cap B_{\eta_0/2}$
with $x_j \to x \in \Reg(u) \cap \bar{B}_{\eta_0/2}$,
and $y_j \in C^+(x_j,\eps)$ with $y_j - x_j \to 0$ such that $u(y_j)=0$.
By Proposition \ref{p:C11}, \eqref{e:decay unico} and \eqref{e:cont blow-up}, 
the rescalings $u_{\LL(x_j),r_j}$, for $r_j =|\LL^{-1}(x_j)(y_j-x_j)|$, 
converge uniformly to $v_x$. %by \eqref{e:holder normal} 
Up to subsequences assume that  
$r_j^{-1}\LL^{-1}(x_j)(y_j-x_j) \to z \in C^+(x,\eps)\cap\mathbb{S}^{n-1}$, 
then $v_x(z) =0$.
This contradicts the fact that $x\in \Reg(u)$ and $v_x>0$ on $C^+(x,\eps)$ thanks to
$f(x)\geq c_0>0$ (cp.~(H3)).
Clearly, we can argue analogously for the second inclusion.

We show next that $\Lambda_u \cap B_{\rho_1}$ is the subgraph of a function $g$, 
for a suitably chosen small $\rho_1>0$. Without loss of generality assume that
$\frac{\A^{-1/2}(\zz) n(\zz)}{|\A^{-1/2}(\zz) n(\zz)|} = e_n$ and set 
\[
g(x'):= \max\{t \in \R \,:\, (x',t) \in \Lambda_u\}
\] 
for all points $x'\in \R^{n-1}$ with $|x'| \leq \delta \sqrt{1-\eps^2}$. Note that 
by \eqref{e:claimconi} this maximum exists and belongs to $[-\eps\delta, \eps\delta]$; 
and moreover the inclusions in \eqref{e:claimconi} imply that $(x',t) \in \Lambda_u$ 
for every $-\eps\,\delta <t<g(x')$, and $(x',t) \in N_u$ for every $g(x')<t <\eps\,\delta$.

Eventually, by taking into account \eqref{e:holder normal}, we conclude that 
$g$ is $C^{1,\beta}$ regular.
\end{proof}

\subsection{Singular free boundary points}
In this section we prove that the singular set of the free boundary is
contained in the countable union of $C^1$ submanifolds.

We recall that, if $x\in \Sing(u)$, then the unique blow-up $v_{x}$ 
is given by
\[
v_{x}(y)=\langle\LL^{-1}(x) \,\mathbb{B}_{x}\, \LL^{-1}(x)y, y\rangle, 
\]
with $\mathbb{B}_{x}$ a symmetric, positive definite matrix satisfying 
$\textup{Tr}(\mathbb{B}_{x}) = \frac{1}2$ (see Proposition~\ref{p:uniform continuity}).
We define the singular strata according to the dimension of the kernel of
$\mathbb{B}_{x}$.

\begin{definition}\label{d:strata}
The singular stratum $S_k$ of dimension $k$, for $k=0, \ldots, n-1$, is the subset
of points $x \in \Sing(u)$ with $\textup{rank} (\mathbb{B}_x) = k$.
\end{definition}

In particular, Theorem~\ref{t:singular} below shows that $S_k$ is ${\mathcal H}^k$ 
rectifiable, and moreover that $\cup_{k=l}^{n-1} S_k$ is a closed set for every 
$l=0, \ldots, n-1$.

\begin{theorem}\label{t:singular}
Let $\zz\in S_k$. Then, there exists $r>0$ such that $S_k \cap B_r$ is
contained in a $C^1$ regular $k$-dimensional submanifold of $\R^n$.
\end{theorem}

\begin{proof}
The proof is divided into two steps. We start off proving the continuity of the map 
\[
\Sing(u)\ni x \mapsto \LL^{-1}(x)\,\mathbb{B}_x\,\LL^{-1}(x).
\]
In turn, by taking this and Proposition~\ref{p:uniform continuity} into account, 
we conclude by means of Whitney's extension theorem and the implicit function 
theorem following \cite{Ca98}. We give the full proof for the sake of completeness. 

To establish the continuity of 
$\Sing(u)\ni x \mapsto \LL^{-1}(x)\,\mathbb{B}_x\,\LL^{-1}(x)$
we argue as in Theorem~\ref{t:regular} by comparing two blow-ups at different 
points. To this aim, note that for some dimensional constant $C>0$
\begin{equation}\label{e:continuita2}
|\LL^{-1}(x)\,\mathbb{B}_x\,\LL^{-1}(x) - 
\LL^{-1}(z)\,\mathbb{B}_z\, \LL^{-1}(z)| 
\leq C\, \|v_x - v_z\|_{L^1(\de B_{1/2})},
\end{equation}
as the right hand side above is a norm on symmetric matrices. 

Fix a compact set $K\subset\Sing(u)$ and let $\sigma_K$ be the modulus of 
continuity in Proposition~\ref{p:uniform continuity}. 
Then, for all $x$ and $z\in K$, setting $s=|x-z|^{1-\theta}\in(0,r_K)$ for 
$\theta=\alpha\wedge\frac12$, we get for some dimensional constant $C>0$ 
\begin{align}\label{e:continuita}
\|v_x - v_z\|_{L^1(\de B_{1/2})}\leq &
\|v_x - u_{x,s}\|_{L^1(\de B_{1/2})} + \|u_{x,s} - u_{z,s}\|_{L^1(\de B_{1/2})} 
+\|u_{z,s} - v_z\|_{L^1(\de B_{1/2})}\notag\\
\stackrel{\eqref{e:decay unico sing}}{\leq}& C\, \sigma_K(|x-z|^{1-\theta}) 
+ C\,|x-z|^\theta,
\end{align}
where the difference of the two rescaled maps is estimated as in the second line of 
inequality \eqref{e:cont blow-up} in Theorem~\ref{t:regular}.
Inequalities \eqref{e:continuita2} and \eqref{e:continuita} establish the required
continuity.

Furthermore, we claim that there exists a function $g\in C^2(\Rn)$ such that for 
all $x\in K$
\begin{equation}\label{e:whitney}
g(y)-v_x(y-x)=o(|y-x|^2)\quad\text{ as }y\to x.
\end{equation}
To this aim we show that the family $v_x(\cdot-x)$, $x\in K$, of translations of the 
blow-ups satisfies the assumptions of Whitney's extension theorem (see \cite{Zie}).
More precisely, we show that the polynomials $p_x(y):=v_x(y-x)$, $x\in K$, satisfies
\begin{itemize}
\item[(i)] $p_x(x)=0$ for all $x\in S_k$,
\item[(ii)]  $D^l(p_x-p_z)(x)=o(|x-z|^{2-l})$ for all $x$ and $z\in K\cap S_k$, 
and $l\in\{0,1,2\}$.
\end{itemize}
Condition (i) is trivially satisfied; instead for what (ii) is concerned, 
we note that estimate \eqref{e:decay unico sing} in 
Proposition~\ref{p:uniform continuity} rewrites, for $r\in(0,\tilde r_K)$
with $\tilde r_K$ depending only on $r_K$ and $\lambda$, as 
\[
\|u-p_z\|_{C^0(B_r(z))}\leq r^2\,\sigma_K(r),\quad\textrm{and}\quad
\|\nabla u-\nabla p_z\|_{C^0(B_r(z))}\leq r\,\sigma_K(r).
\]
Therefore, since $u(x)=0$ and $\nabla u(x)=\zz$ imply
\[
|p_x(x)-p_z(x)| = |u(x)-p_z(x)| \quad\text{and}\quad
|\nabla p_x(x)- \nabla p_z(x)| = |\nabla u(x)-\nabla p_z(x)|,
\]
then (ii) is verified for $l\in\{0,1\}$.
In addition, if $l=2$, condition (ii) reduces to the continuity of 
the map $\Sing(u) \ni x \mapsto f(x)\, \A^{-1/2}(x)\,\mathbb{B}_x\,\A^{-1/2}(x)$ 
established above.

Equality \eqref{e:whitney} gives that $K\subseteq\{\nabla g=\zz\}$. 
Suppose now that $\zz\in K\cap S_k$, and arrange the coordinates of $\Rn$ in 
a way that $e_i$, $i\in\{1,\ldots,n-k\}$, are the eigenvalues of $\nabla^2 g(\zz)$. 
Then, the $(n-k)\times(n-k)$ minor of $\nabla^2 g(\zz)$ composed by the first $n-k$ 
rows and columns, is not null. 
Therefore, the implicit function theorem yields that $\cap_{i=1}^{n-k}\{\partial_ig=0\}$ 
is a $C^1$ submanifold in a neighborhood of $\zz$, and the conclusion follows at once 
noting that $K\cap S_k\subseteq
\{\nabla g=\zz\}\subseteq\cap_{i=1}^{n-k}\{\partial_ig=0\}$.
\end{proof}


\begin{thebibliography}{99}

\bibitem{ACF} Alt, H.W.; Caffarelli, L.A.; Friedman, A.,
Variational problems with two phases and their free
boundaries.
Trans. Amer. Math. Soc., 282 (1984), no.2, 431--461.


\bibitem{AtCa} Athanasopoulos, I.; Caffarelli, L.A.,
A theorem of real analysis and its application to free boundary
problems.
Comm. Pure Appl. Math., 38 (1985), no. 5, 499--502.

\bibitem{Ca77}  Caffarelli, L. A., 
The regularity of free boundaries in higher dimensions. 
Acta Math. 139 (1977), no. 3-4, 155--184.

\bibitem{Ca80}  Caffarelli, L. A., 
Compactness methods in free boundary problems. 
Comm. Partial Differential Equations 5 (1980), no. 4, 427--448.

\bibitem{Ca98-Fermi}  Caffarelli, L. A., 
The obstacle problem revisited. 
Lezioni Fermiane. [Fermi Lectures] Accademia Nazionale dei Lincei, 
Rome; Scuola Normale Superiore, Pisa, 1998. ii+54 pp.

\bibitem{Ca98}  Caffarelli, L. A., 
The obstacle problem revisited. 
J. Fourier Anal. Appl. 4 (1998), no. 4-5, 383--402.

\bibitem{CFMS} Caffarelli, L.A.; Fabes, E.; Mortola, S.; Salsa, S.,
Boundary behavior of nonnegative solutions of
elliptic operators in divergence form. 
Indiana Univ. Math. J., 30 (1981), no. 4, 621--640.

\bibitem{CS} Caffarelli, L.A.; Salsa, S.,
A geometric approach to free boundary problems. Graduate Studies in Mathematics, 68. 
American Mathematical Society, Providence, RI, 2005. x+270 pp.

\bibitem{CKS} Caffarelli, L.A.; Karp, L.; Shahgholian, H.,
Regularity of a free boundary with application to the Pompeiu problem. 
Ann. of Math. (2) 151 (2000), no. 1, 269--292. 

% \bibitem{CaRi}  Caffarelli, L. A.; Rivi\'ere, N. M., 
% Asymptotic behaviour of free boundaries at their singular points. 
% Ann. of Math. (2) 106 (1977), no. 2, 309--317.

\bibitem{CFS} Cerutti, M.C.; Ferrari, F.; Salsa, S.,
Two-phase problems for linear elliptic operators with variable coefficients: Lipschitz free boundaries are $C^{1,\gamma}$. 
Arch. Ration. Mech. Anal. 171 (2004), no. 3, 329--348.

\bibitem{FeSa07} Ferrari, F.; Salsa, S., 
Regularity of the free boundary in two-phase problems for linear elliptic operators. 
Adv. Math. 214 (2007), no. 1, 288--322. 

\bibitem{FeSa10}  Ferrari, F.; Salsa, S., 
Regularity of the solutions for parabolic two-phase free boundary problems. 
Comm. Partial Differential Equations 35 (2010), no. 6, 1095--1129. 

\bibitem{Frie} Friedman, A.,
Variational Principles and Free Boundary Problems.
Second edition. Robert E. Krieger Publishing Co., Inc., Malabar, FL, 1988. x+710 pp. 

\bibitem{Gi} Giusti, E.,
Equazioni ellittiche del secondo ordine,
Quaderni Unione Matematica Italiana 6, Pitagora Editrice, Bologna, 1978.
v+213 pp.

\bibitem{GT}  Gilbarg, D.; Trudinger, N.S., 
Elliptic partial differential equations of second order. 
Reprint of the 1998 edition. Classics in Mathematics. 
Springer-Verlag, Berlin, 2001. xiv+517 pp.

\bibitem{JK} Jerison, D.; Kenig, C.,
Boundary behaviour of harmonic functions in nontangentially accessible
domains. Adv. in Math., 46 (1982), no. 1, 80--147.

\bibitem{KS}  Kinderlehrer, D.;  Stampacchia, G., 
An introduction to variational inequalities and their applications. 
Pure and Applied Mathematics, 88. Academic Press, Inc. 
[Harcourt Brace Jovanovich, Publishers], New York-London, 1980. xiv+313 pp.

\bibitem{Kukavica}
Kukavica, I., Quantitative uniqueness for second-order elliptic operators.  
Duke Math. J.  91  (1998),  no. 2, 225--240. 

\bibitem{Lin} Lin, F.,
On regularity and singularity of free boundaries in obstacle problems. 
Chin. Ann. Math. Ser. B 30 (2009), no. 5, 645--652.

\bibitem{MaPe}  Matevosyan, N.; Petrosyan, A., 
Almost monotonicity formulas for elliptic and parabolic operators with variable 
coefficients. Comm. Pure Appl. Math. 64 (2011), no. 2, 271--311.

\bibitem{Monneau} Monneau, R., 
On the number of singularities for the obstacle problem in two dimensions.  
J. Geom. Anal.  13  (2003),  no. 2, 359--389.

\bibitem{PeSh}  Petrosyan, A.; Shahgholian, H., 
Geometric and energetic criteria for the free boundary regularity in an obstacle-type 
problem. Amer. J. Math. 129 (2007), no. 6, 1659--1688.

\bibitem{PSU}  Petrosyan, A.; Shahgholian, H.; Uraltseva, N., 
Regularity of free boundaries in obstacle-type problems. 
Graduate Studies in Mathematics, 136. American Mathematical Society, 
Providence, RI, 2012. x+221 pp.

\bibitem{Ro} Rodrigues, J.F.,
Obstacle problems in mathematical physics. North-Holland Mathematics Studies, 134. 
Notas de Matemática [Mathematical Notes], 114. North-Holland Publishing Co., Amsterdam, 1987. xvi+352 pp.
%\bibitem{St} 

\bibitem{Wang00} Wang, P.Y., 
Regularity of free boundaries of two-phase problems for fully nonlinear elliptic equations of second order. I. 
Lipschitz free boundaries are $C^{1,\alpha}$. Comm. Pure Appl. Math. 53 (2000), no. 7, 799--810.

\bibitem{Wang02} Wang, P.Y., 
Regularity of free boundaries of two-phase problems for fully nonlinear elliptic equations of second order. II. 
Flat free boundaries are Lipschitz. 
Comm. Partial Differential Equations 27 (2002), no. 7-8, 1497--1514.


\bibitem{Weiss} Weiss, G.S., 
A homogeneity improvement approach to the obstacle problem. 
Invent. Math.  138  (1999),  no. 1, 23--50.

\bibitem{Weiss01} Weiss, G.S., 
An obstacle-problem-like equation with two phases: pointwise regularity of the solution 
and an estimate of the Hausdorff dimension of the free boundary. 
Interfaces Free Bound. 3 (2001), no. 2, 121--128.

\bibitem{Zie} Ziemer, W.P.,
Weakly differentiable functions.
Sobolev spaces and functions of bounded variation.
GTM, 120. Springer-Verlag, New York, 1989. xvi+308 pp.











\end{thebibliography}
\end{document}